\documentclass[a4paper,reqno,11pt]{amsart}
\usepackage{amssymb,amstext,amsmath,amscd,amsthm,amsfonts,enumerate,graphicx,latexsym,mathrsfs}
\usepackage{paralist}
\usepackage[usenames]{color}
\usepackage[all]{xy}
\makeatletter

\@addtoreset{equation}{section}
\makeatother
\newtheorem{theorem}{Theorem}[section]
\newtheorem*{theorem*}{Main Theorem}
\newtheorem{corollary}[theorem]{Corollary}
\newtheorem{lemma}[theorem]{Lemma}
\newtheorem{proposition}[theorem]{Proposition}

\newtheorem{definition-proposition}[theorem]{Definition-Proposition}
\theoremstyle{definition}
\newtheorem{definition}[theorem]{Definition}
\newtheorem{remark}[theorem]{Remark}
\newtheorem{condition}[theorem]{Condition}

\newtheorem*{question*}{Question}
\newtheorem{question**}{Question}

\newtheorem*{conjecture*}{Conjecture}
\newtheorem{example}[theorem]{Example}

\newtheorem*{claim*}{Claim}
\newtheorem{mainthm}{Theorem}

\def\Mor{\operatorname{Mor}}\def\Ob{\operatorname{Ob}}

\def\Ext{\operatorname{Ext}}

\def\Hom{\operatorname{Hom}}
\def\mod{\operatorname{\mathsf{mod}}}
\def\proj{\operatorname{\mathsf{P}}}\def\inj{\operatorname{\mathsf{I}}}

\def\Ab{\mathsf{Ab}}
\def\add{\operatorname{\mathsf{add}}}

\newcommand{\op}{\mathsf{op}}
\newcommand{\A}{\mathcal{A}}
\newcommand{\C}{\mathcal{C}}\newcommand{\D}{\mathcal{D}}

\renewcommand{\H}{\mathcal{H}}\newcommand{\I}{\mathcal{I}}
\newcommand{\K}{\mathcal{K}}
\newcommand{\N}{\mathcal{N}}
\renewcommand{\P}{\mathcal{P}}\newcommand{\Q}{\mathcal{Q}}
\renewcommand{\S}{\mathcal{S}}\newcommand{\T}{\mathcal{T}}
\newcommand{\U}{\mathcal{U}}\newcommand{\V}{\mathcal{V}}
\newcommand{\W}{\mathcal{W}}\newcommand{\Z}{\mathcal{Z}}
\newcommand{\End}{\operatorname{End}}
\newcommand{\Ker}{\operatorname{Ker}}
\renewcommand{\Im}{\operatorname{Im}}
\newcommand{\Cok}{\operatorname{Cok}}

\newcommand{\xto}{\xrightarrow}

\def\cone{\operatorname{\mathsf{Cone}}}
\def\cocone{\operatorname{\mathsf{CoCone}}}
\def\udef{\operatorname{\U-\mathsf{def}}}\def\tinf{\operatorname{\T-\mathsf{inf}}}

\def\wfib{\operatorname{\mathsf{wFib}}}\def\wcof{\operatorname{\mathsf{wCof}}}
\begin{document}
\setlength{\baselineskip}{15pt}
\title[Abelian categories from triangulated categories]{Abelian categories from triangulated categories via Nakaoka-Palu's localization}
\author{Yasuaki Ogawa}
\email{ogawa.yasuaki.gh@cc.nara-edu.ac.jp} %
\address{Center for Educational Research of Science and Mathematics, Nara University of Education, Takabatake-cho, Nara, 630-8528, Japan}
\begin{abstract}
The aim of this paper is to provide an expansion of Abe-Nakaoka's heart construction to the following two different realizations of the module category over the endomorphism ring of a rigid object in a triangulated category:
Buan-Marsh's localization and Iyama-Yoshino's subfactor.
Our method depends on a modification of Nakaoka-Palu's HTCP localization, a Gabriel-Zisman localization of extriangulated categories which is also realized as a subfactor of the original ones.
Besides of the heart construction, our generalized HTCP localization involves the following phenomena:
(1) stable category with respect to a class of objects;
(2) recollement of triangulated categories;
(3) recollement of abelian categories under a certain assumption.
\end{abstract}
\maketitle


\section*{Introduction}\label{sec:intro}
In many fields of mathematics containing representation theory, it is basic to study how to construct related abelian categories from a given triangulated category $\C$.
There are many researches in this context, for example \cite{BBD, KR07, KZ08, Nak11, AN12, BM13a, BM13b, Bel13, Nak13, Pal14, LN19, HS20}.
Our study has its origin in the following result in \cite[Thm. 3.3]{KZ08}: If $T$ is a cluster-tilting object in a triangulated category $\C$, then the additive quotient $\C/\add T$ is equivalent to $\mod\End_\C(T)^{\op}$ the module category of finitely presented right $\End_\C(T)^{\op}$-modules (see also \cite{KR07}).

Afterwards, Buan and Marsh pointed out that, as long as $T$ is a rigid object in $\C$, there exists a localization functor from $\C$ to $\mod\End_\C(T)^{\op}$, generalizing Koenig-Zhu's set-up.
More precisely, in \cite{BM13a}, they firstly found a class $\mathbb{S}$ of morphisms with an explicit description such that the associated localization $\mathsf{L}_\mathbb{S}:\C\to\C[\mathbb{S}^{-1}]$ induces an equivalence $\C[\mathbb{S}^{-1}]\simeq \mod\End_\C(T)^{\op}$.
Following this, in \cite{BM13b}, they factorized the above localization $\mathsf{L}_\mathbb{S}$ as the composition of two localizations:
\begin{equation}\label{Buan-Marsh's_localization}
\xymatrix{
\C\ar[rd]_{\varpi}\ar[rr]^{\overset{\textnormal{\cite{BM13a}}}{\mathsf{L}_\mathbb{S}}}&&\mod\End_\C(T)^{\op}\\
&\C/(T^{\perp_1})\ar[ru]_{\underset{\textnormal{\cite{BM13b}}}{\mathsf{Loc}}}&
}
\end{equation}
where $\varpi$ denotes an additive quotient.
Under additional assumptions on $\C$, there are many advantages of this factorization, since $\C/(T^{\perp_1})$ is preabelian and the second one $\mathsf{Loc}$ is a Gabriel-Zisman localization admitting a calculus of left and right fractions.
Their construction of preabelian categories was improved by many authers, e.g. \cite{Bel13, Nak13, Liu13, LN19, HS20}.
On the other hand, before \cite{BM13a, BM13b}, the category $\mod\End_\C(T)^{\op}$ has already been realized as a subfactor of $\C$ in \cite{IY08}, namely, there exists an equivalence 
\begin{equation}\label{Iyama-Yoshino's_subfactor}
\frac{\add (T[-1])*\add T}{\add T}\xto{\sim}\mod\End_\C(T)^{\op},
\end{equation}
where $[-1]$ is a desuspension of $\C$.

As another generalization of Koenig-Zhu's construction, Abe and Nakaoka provided a new construction of a related abelian category $\H/\W$ from a given triangulated category $\C$ equipped with a cotorsion pair $(\S,\V)$ in \cite{AN12}.
Their method also generalizes the construction of the heart of a $t$-structure proved in \cite{BBD} (see \cite[Prop. 2.6]{Nak11} for details).
So their abelian category $\H/\W$ is still called the heart of $(\S,\V)$.
As we will see in Subsection \ref{ssec:from_CP_to_gHTCP}, Abe-Nakaoka's heart construction is a generalization of Iyama-Yoshino's subfactor (\ref{Iyama-Yoshino's_subfactor}).
So it is natural to ask whether the heart $\H/\W$ can be realized as a nice Gabriel-Zisman localization of $\C$.
This question has already arisen in \cite[Section 6]{BM13a} and \cite[Section 6]{BM13b}, and some answers were obtained (see Section \ref{sec:related_work}).

The aim of this article is to give a more complete answer to Buan-Marsh's question in connection with Iyama-Yoshino's subfactor, which unifies and improves some related results (e.g. \cite[Thm. 6.3]{Nak13}, \cite[Thm. 5.6]{HS20}), where the method completely differs from them.
Our method is a modification of Nakaoka-Palu's localization via Hovey twin cotorsion pair (HTCP localization) which was inspired from Hovey triple \cite{Hov02, Gil11} and introduced in \cite{Nak18, NP19}.
They firstly introduced a notion of extriangulated category which is a simultaneous generalization of exact category and triangulated one.
The HTCP localization turns out to be a Gabriel-Zisman localization of an extriangulated category which can be realized as a subfactor of the original one, and covers many important phenomena, e.g. recollement of triangulated categories, Happel's (projectively) stable category of a Frobenius category \cite{Hap88} and Iyama-Yoshino's triangulated structure via mutation pairs \cite{IY08}.
Our generalized HTCP localization is still a Gabriel-Zisman localization of an extriangulated category which is equivalent to a subfactor.
It covers a wider class of localizations containing stable category with respect to a class of objects and recollement of abelian categories under a certain assumption.

\begin{mainthm}[Theorem \ref{thm:RL=LR_implies_universality}]\label{mainthm:A}
Let $((\S,\T),(\U,\V))$ be a generalized HTCP in an extriangulated category $\C$.
Then, there exists a class $\mathbb{V}$ of morphisms in $\C$ such that the associated Gabriel-Zisman localization $\mathsf{L}_\mathbb{V}:\C\to\C[\mathbb{V}^{-1}]$ induces an equivalence
$$\Phi:\frac{\T\cap\U}{\S\cap\V}\xto{\sim}\C[\mathbb{V}^{-1}].$$
\end{mainthm}

Note that there are two similar pictures below deduced from different set-ups in (\ref{Buan-Marsh's_localization}), (\ref{Iyama-Yoshino's_subfactor}) and Theorem \ref{mainthm:A}:
\begin{equation}\label{picture_of_this_article}
\xymatrix@!C=36pt@R=32pt{
&T\textnormal{:rigid}\ar@{~>}[rd]^{\textnormal{Buan-Marsh}}\ar@{~>}[ld]_{\textnormal{Iyama-Yoshino}}&&&((\S,\T),(\U,\V))\ar@{~>}[rd]\ar@{~>}[ld]&\\
\frac{\add T[-1]*\add T}{\add T}\ar@{->}[rr]^\sim&&\C[\mathbb{S}^{-1}]&\frac{\T\cap\U}{\S\cap\V}\ar@{->}[rr]^\sim&&\C[\mathbb{V}^{-1}]
}
\end{equation}
This article was motivated by the above analogy.
Our second result shows that the heart construction can be regarded as a generalized HTCP localization.

\begin{mainthm}[Corollary \ref{cor:heart_is_gHTCP}]
\label{mainthm:B}
Let $\C$ be a triangulated category equipped with a cotorsion pair $(\S,\V)$.
Then, there exists a class $\mathbb{V}$ of morphisms in $\C$ such that the associated localization $\mathsf{L}_\mathbb{V}:\C\to\C[\mathbb{V}^{-1}]$ induces an equivalence $\Phi:\H/\W\xto{\sim}\C[\mathbb{V}^{-1}]$.
\end{mainthm}

Furthermore, we factorize the localization $\mathsf{L}_\mathbb{V}$ into nice localizations, following Buan-Marsh's investigation.

\begin{mainthm}[Theorem \ref{thm:generalized_BM2}]
\label{mainthm:C}
Let $\C$ be the above.
Assume that $\S*\V$ is functorially finite.
Then,
\begin{enumerate}
\item[\textnormal{(1)}] The additive quotient $\C/\add(\S*\V)$ is preabelian and the class $\mathbb{R}$ of regular morphisms in it admits a calculus of left and right fractions;
\item[\textnormal{(2)}] The localization $\mathsf{L}_\mathbb{V}:\C\to\H/\W$ is factored as the composition of the additive quotient $\varpi:\C\to\C/\add(\S*\V)$ and the Gabriel-Zisman localization $\mathsf{L}_\mathbb{R}: \C/\add(\S*\V)\to \H/\W$ with respect to the class $\mathbb{R}$.
\end{enumerate}
\end{mainthm}

This article is organized as follows:
Section \ref{sec:preliminary} will be devoted to prepare basic results on extriangulated category, Gabriel-Zisman localization and preabelian category.
Section \ref{sec:gHTCP} contains the first main result of this article.
Here, we formulate a generalized HTCP localization and prove Theorem \ref{mainthm:A}.
In Section \ref{sec:heart}, using the generalized HTCP, we investigate two different  aspects of the heart in a triangulated category $\C$, i.e., a subfactor and a localization of $\C$, and then Theorems \ref{mainthm:B} and \ref{mainthm:C} are proved.
Finally, in Section \ref{sec:related_work}, we mention another related approach.

\subsection*{Notation and convention}
The symbol $\C$ always denotes a category, and the set of morphisms $X\rightarrow Y$ in $\C$ is denoted by $\C(X,Y)$ or simply denoted by $(X,Y)$ if there is no confusion.
The class of objects (resp. morphisms) in $\C$ is denoted by $\Ob\C$ (resp. $\Mor\C$).
We denote by $\C^{\op}$ the opposite category.
If there exists a fully faithful functor $\N\hookrightarrow\C$,
we often regard $\N$ as a full subcategory of $\C$.
If a given category $\C$ is additive, its subcategory $\U$ is always assumed to be full, additive and closed under isomorphisms, and we only consider additive functors between them.
For $X\in\C$, if $\C(U,X)=0$ for any $U\in\U$, we write abbreviately $\C(\U,X)=0$.
Similar notations will be used in many places.
For an additive functor $F:\C\rightarrow \D$,
we define the \textit{image} and \textit{kernel} of $F$ as the full subcategories
$$\Im F:=\{Y\in\D\mid {^\exists X\in\C},\  F(X)\cong Y\}\ \ \textnormal{and}\ \  \Ker F:=\{X\in\C\mid F(X)=0\},$$
respectively.
For a full subcategory $\U$ in $\C$, the symbol $F|_{\U}$ denotes the restriction of $F$ on $\U$.

\section{Preliminary}\label{sec:preliminary}

\subsection{Extriangulated category}\label{ssec:extriagulated_category}
In this section, we recall results necessary for our purpose and terminology on extriangulated categories.
An \textit{extriangulated category} is defined to be a triple $(\C,\mathbb{E},\mathfrak{s})$ of
\begin{itemize}
\item an additive category $\C$;
\item an additive bifunctor $\mathbb{E}:\C^{\op}\times\C\to \Ab$, where $\Ab$ is the category of abelian groups;
\item a correspondence $\mathfrak{s}$ which associates each equivalence class of  a sequence  of the form $X\to Y\to Z$ in $\C$ to an element in $\mathbb{E}(Z,X)$ for any $Z,X\in\C$,
\end{itemize}
which satisfies some `additivity' and `compatibility'.
It is simply denoted by $\C$ if there is no confusion.
We refer to \cite[Section 2]{NP19} for its detailed definition.
An extriangulated category was introduced to unify triangulated category and exact one.
More precisely, by putting $\mathbb{E}:=\C(-,-[1])$, a triangulated category $(\C,[1],\Delta)$ can be regarded as an extriangulated category \cite[Prop. 3.22]{NP19}.
By putting $\mathbb{E}:=\Ext^1_\C(-,-)$, an exact category $\C$ can be regarded as an extriangulated category \cite[Example 2.13]{NP19}.
We shall use the following terminology in many places.

\begin{definition}
Let $(\C,\mathbb{E},\mathfrak{s})$ be an extriangulated category.
\begin{enumerate}
\item[\textnormal{(1)}] We call an element $\delta\in\mathbb{E}(Z, X)$ an \textit{$\mathbb{E}$-extension}, for any $X,Z\in\C$;
\item[\textnormal{(2)}] A sequence $X\xto{f}Y\xto{g}Z$ corresponding to an $\mathbb{E}$-extension $\delta\in\mathbb{E}(Z,X)$ is called a \textit{conflation}.
 In addition, $f$ and $g$ are called an \textit{inflation} and a \textit{deflation}, respectively.
\item[\textnormal(3)] An object $P\in\C$ is said to be \textit{projective} if for any deflation $g:Y\to Z$, the induced morphism $\C(P,g):\C(P,Y)\to\C(P,Z)$ is surjective.
We denote by $\proj(\C)$ the subcategory of projectives in $\C$.
An \textit{injective} object and $\inj(\C)$ are defined dually.
\item[\textnormal(4)] We say that \textit{$\C$ has enough projectives} if for any $Z\in\C$, there exists a conflation $X\to P\to Z$ with $P$ projective.
\textit{Having enough injectives} is defined dually.
\end{enumerate}
\end{definition}

Keeping in mind the triangulated case, we introduce the notions cone and cocone.

\begin{proposition}
Let $\C$ be an extriangulated category.
For an inflation $f\in\C(X,Y)$, take a conflation $X\xto{f} Y\xto{g} Z$, and
denote this $Z$ by $\cone(f)$.
We call $\cone(f)$ a \textnormal{cone} of $f$.
Similarly, we denote the object $X$ by $\cocone(g)$ and call it a \emph{cocone} of $g$.
For a given inflation $f$ (resp. deflation $g$), $\cone(f)$ (resp. $\cocone(g)$) is uniquely determined up to isomorphism.
\end{proposition}

Furthermore, for any subcategories $\U$ and $\V$ in $\C$, we define a full subcategory $\cone(\V,\U)$ to be the one consisting of objects $X$ appearing in a conflation $V\to U\to X$ with $U\in\U$ and $V\in\V$.
A subcategory $\cocone(\V,\U)$ is defined dually.
Next, we recall the notion of weak pullback in extriangulated categories.
Consider a conflation $X\xto{f} Y\xto{g} Z$ corresponding to $\mathbb{E}$-extension $\delta\in\mathbb{E}(Z,X)$ and a morphism $z:Z'\to Z$.
Put $\delta':=\mathbb{E}(z,X)(\delta)$ and consider a corresponding conflation $X\to E\to Z'$.
Then, there exists a commutative diagram of the following shape
$$
\xymatrix{
X\ar[r]\ar@{=}[d]&E\ar@{}[rd]|{\textnormal{(Pb)}}\ar[r]\ar[d]&Z'\ar[d]^{z}\\
X\ar[r]^f&Y\ar[r]^g&Z
}
$$
The commutative square (Pb) is called a \textit{weak pullback of $g$ along $z$} which is a generalization of the pullback in exact categories and the homotopy pullback in triangulated categories.
The dual notion \textit{weak pushout} exists and it will be denoted by (Po) (see \cite[Cor. 3.16]{NP19} for details).

We end the section by mentioning that the class of extriangulated categories is closed under certain operations.

\begin{proposition}\cite[Rem. 2.18, Prop. 3.30]{NP19}\label{prop:closed_under_operation}
Let $\C$ be an extriangulated category.
\begin{enumerate}
\item[\textnormal{(1)}] Any extension-closed subcategory admits an extriangulated structure induced from that of $\C$.
\item[\textnormal{(2)}] Let $\I$ be a full additive subcategory, closed under
isomorphisms which satisfies $\I\subseteq\proj(\C)\cap\inj(\C)$, then the additive quotient $\C/\I$ has an extriangulated structure, induced from that of $\C$ (see Example \ref{ex:univ_of_stable_cat} for the definition of additive quotient).
\end{enumerate}

\end{proposition}

\subsection{Gabriel-Zisman localization}\label{ssec:Gabriel-Zisman_localization}
Since we are interested in Gabriel-Zisman localizations of extriangulated categories, we recall its definition, following \cite{Fri08} (see also \cite{GZ67}).

\begin{definition}
Let $\C$ and $\D$ be categories and $\mathbb{S}$ a class of morphisms in $\C$.
A functor $\mathsf{L}_\mathbb{S}:\C\to\D$ is called a \textit{Gabriel-Zisman localization of $\C$ with respect to $\mathbb{S}$} if the following universality holds:
\begin{enumerate}
\item[\textnormal{(1)}] $\mathsf{L}_\mathbb{S}(s)$ is an isomorphism in $\D$ for any $s\in\mathbb{S}$;
\item[\textnormal{(2)}] For any functor $F:\C\to\D'$ which sends each morphism in $\mathbb{S}$ to an isomorphism in $\D'$, there uniquely, up to isomorphism, exists a functor $F':\D\to\D'$ such that $F\cong F'\circ \mathsf{L}_\mathbb{S}$.
\end{enumerate}
In this case, we denote by $\mathsf{L}_\mathbb{S}:\C\to\C[\mathbb{S}^{-1}]$.
\end{definition}

The Gabriel-Zisman localization $\C[\mathbb{S}^{-1}]$ always exists provided there are no set-theoretic obstructions.
So, whenever we consider the Gabriel-Zisman localization of $\C$, we assume that $\C$ is skeletally small.
Morphisms in the new category $\C[\mathbb{S}^{-1}]$ can be regarded as compositions of the original morphisms and the formal inverses.
We refer to \cite[Thm. 2.1]{Fri08} for an explicit construction of $\C[\mathbb{S}^{-1}]$.
If the class $\mathbb{S}$ satisfies the following conditions and its dual ones, 
any morphism in $\C[\mathbb{S}^{-1}]$ has a very nice description, see \cite[Section I.2]{GZ67} for details.

\begin{definition}
Let $\mathbb{S}$ be a class of morphisms in $\C$.
\begin{enumerate}
\item[(RF1)] The identity morphisms of $\C$ lie in $\mathbb{S}$ and $\mathbb{S}$ is closed under composition.
\item[(RF2)] Any diagram of the form:
$$
\xymatrix@R=16pt@C=16pt{
&B\ar[d]^f\\
C\ar[r]^s&D
}
$$
with $s\in\mathbb{S}$ can be completed in a commutative square of the form:
$$
\xymatrix@R=16pt@C=16pt{
A\ar[r]^{s'}\ar[d]_{f'}&B\ar[d]^f\\
C\ar[r]^s&D
}
$$
with $s'\in\mathbb{S}$.
\item[(RF3)] If $s:Y\to Y'$ in $\mathbb{S}$ and $f,f':X\to Y$ are morphisms such that $sf=sf'$, then there exists $s':X'\to X$ in $\mathbb{S}$ such that $fs'=f's'$.
\end{enumerate}
If the above and the dual conditions are satisfied, we say that $\mathbb{S}$ or $\mathsf{L}_\mathbb{S}$ \textit{admits a calculus of left and right fractions}.
\end{definition}

As is well-known, the Verdier localization of a triangulated category $\C$ with respect to its thick subcategory $\N$ is defined to be a Gabriel-Zisman localization which admits a calculus of left and right fractions.
More precisely, it is the Gabriel-Zisman localization of $\C$ with respect to the class of morphisms whose cones belong to $\N$ (e.g. \cite{Nee01}).
Similarly, the Serre localization of an abelian category is also an example of such Gabriel-Zisman localizations (e.g. \cite{Pop73}).
The following example will play an important role in Section \ref{sec:gHTCP}.

\begin{example}\label{ex:univ_of_stable_cat}
Let $\C$ be an additive category and $\I$ its full subcategory closed under direct summands.
We define the \textit{stable category $\C/\I$ of $\C$ with respect to $\I$} as the ideal quotient of $\C$ modulo the (two-sided) ideal in $\C$ consisting of all morphisms having a factorization through an object in $\I$.
This is also called the \textit{additive quotient of $\C$ with respect to $\I$}.
Consider the class $\mathbb{S}$ of all sections whose cokernels belonging to $\I$, i.e., the morphisms $s$ appearing in some splitting short exact sequence of the form $0\to X\xto{s} X\oplus I\to I\to 0$ with $I\in\I$.
Then, $\C/\I$ is equivalent to the Gabriel-Zisman localization $\C[\mathbb{S}^{-1}]$.
\end{example}
\begin{proof}
This is well-known for experts, but we can not find proper references.
So we include a detailed proof here.

It suffices to show that the additive quotient $\pi:\C\to\C/\I$ has the same universality as that of $\mathsf{L}_\mathbb{S}:\C\to\C[\mathbb{S}^{-1}]$.
To this end, we consider a functor $F:\C\to\D$ which sends any morphisms of $\mathbb{S}$ to isomorphisms in $\D$.
Note that we do not assume that $F$ is additive.

(1) It is obvious that $\pi(s)$ is an isomorphism in $\C/\I$ whenever $s$ belongs to $\mathbb{S}$.

(2) We shall show the existence of a functor $F':\C/\I\to\D$ with $F\cong F'\circ \pi$ and its uniqueness up to isomorphism.
Define such a functor $F'$ as follows: For any $X\in\Ob(\C/\I)=\Ob\C$, we set $F'(X):=F(X)$; For any morphism $\pi(f):X\to Y$ in $\C/\I$, we set $F'(\pi(f)):=F(f)$.
To check the well-definedness, we assume that $\pi(f)=\pi(g)$ where $f, g:X\to Y$ are morphisms in $\C$, that is,
the morphism $f-g$ factors through an object $I\in\I$ as $f-g:X\xto{a}I\xto{b}Y$.
Then we get $f=\begin{pmatrix}
g&b
\end{pmatrix}
\begin{pmatrix}
1\\
a
\end{pmatrix}
$
and consider the following diagram in $\C$:
$$
\xymatrix{
X\ar[r]^{\tiny\begin{pmatrix}
1\\
a
\end{pmatrix}}\ar@{=}[rd]&X\oplus I\ar[r]^{\tiny\begin{pmatrix}
g&b
\end{pmatrix}}&Y\\
&X\ar[ur]_g\ar[u]_{\tiny\begin{pmatrix}
1\\
0
\end{pmatrix}}&
}
$$
Notice that the left half of the above diagram is not commutative.
However, applying $F$ makes it commute.
In fact, the projection $
\tiny\begin{pmatrix}
1&0
\end{pmatrix}
:X\oplus I\to X$
is a left inverse of $\tiny\begin{pmatrix}
1\\
a
\end{pmatrix}$ and
$\tiny\begin{pmatrix}
1\\
0
\end{pmatrix}$.
Since both $F\tiny\begin{pmatrix}
1\\
a
\end{pmatrix}$ and $F\tiny\begin{pmatrix}
1\\
0
\end{pmatrix}$
are isomorphisms in $\D$,
$F\tiny\begin{pmatrix}
1&0
\end{pmatrix}$ is an inverse of them.
It guarantees $F\tiny\begin{pmatrix}
1\\
a
\end{pmatrix}=F\tiny\begin{pmatrix}
1\\
0
\end{pmatrix}$ in $\D$.
Hence we have a desired equality
$$F(f)=F\tiny\begin{pmatrix}
g&b
\end{pmatrix}\circ F\tiny\begin{pmatrix}
1\\
a
\end{pmatrix}=
F\tiny\begin{pmatrix}
g&b
\end{pmatrix}\circ F\tiny\begin{pmatrix}
1\\
0
\end{pmatrix}=F(g).
$$
The commutaivity $F=F'\circ\pi$ automatically holds.

Since $\pi$ is full and dense, the uniqueness of $F'$ is satisfied. This finishes the proof.
\end{proof}

\subsection{Preabelian category}\label{ssec:preabelian_category}
We recall some basic properties of preabelian categories which will be used in Subsection \ref{ssec:factorization}, following \cite{Rum01} (see also \cite{BM13a}).
An additive category $\A$ is called a \textit{preabelian category} if any morphism has a kernel and a cokernel.
A morphism is said to be \textit{regular} if it is both an epimorphism and a monomorphism.
Note that pullbacks and pushouts always exist in a preabelian category $\A$.
In fact, for morphisms $C\xto{d}D\xleftarrow{c}B$, we take the kernel sequence $A\to B\oplus C\xto{(c\ -d)} D$ to obtain the pullback of $d$ along $c$:
\begin{equation}\label{diag:pullback}
\xymatrix@R=16pt@C=16pt{
A\ar@{}[rd]|{\textnormal{(Pb)}}\ar[d]_{b}\ar[r]^{a}&B\ar[d]^{c}\\
C\ar[r]_{d}&D
}
\end{equation}

\begin{definition}
A preabelian category $\A$ is said to be \textit{left integral} if
for any pullback (\ref{diag:pullback}), $a$ is an epimorphism whenever $d$ is an epimorphism.
The \textit{right integrality} is defined dually.
An \textit{integral} category is defined as a preabelian category which is both left integral and right integral.
\end{definition}

The following provides a nice connection between integral categories and abelian categories.

\begin{proposition}\label{prop:localization_of_integral}\cite[p. 173]{Rum01}
For an integral category $\A$, the following hold.
\begin{enumerate}
\item[\textnormal{(1)}] The class $\mathbb{R}$ of regular morphisms admits a calculus of left and right fractions.
\item[\textnormal{(2)}] The localization functor $\mathsf{L}_\mathbb{R}:\A\to\A[\mathbb{R}^{-1}]$ is additive.
\item[\textnormal{(3)}] The category $\A[\mathbb{R}^{-1}]$ is abelian.
\end{enumerate}
\end{proposition}

\section{Localization via generalized Hovey twin cotorsion pair (gHTCP)}\label{sec:gHTCP}

\subsection{Definition of gHTCP}\label{ssec:def_of_gHTCP}
Throughout this section, the symbol $\C=(\C,\mathbb{E},\mathfrak{s})$ is an extriangulated category.
For a subcategory $\U\subseteq \C$ and an object $X\in\C$, a morphism $U\to X$ from an object $U\in\U$ is called a \emph{right $\U$-approximation} of $X$ if the induced morphism $\C(U', U)\to \C(U', X)$ is an epimorphism for any $U'\in\U$.
The dual notion \emph{left $\U$-approximation} exists.

Let $(\U,\V)$ be a pair of full subcategories in $\C$ which are closed under isomorphisms and direct summands.
We introduce the notion of \emph{right/left cotorsion pair} as follows.

\begin{definition}\label{def:u-deflation}
Let $X$ be an object in $\C$.
A right $\U$-approximation $p_X:U_X\to X$ of $X$ is said to be a \emph{$\U$-deflation} of $X$ if it is a deflation and $\cocone(p_X)\in\V$.
Dually, a left $\V$-approximation $\iota_X:X\to V^X$ of $X$ is said to be a \emph{$\V$-inflation} of $X$ if it is an inflation and $\cone(\iota_X)\in\U$.
\end{definition}

\begin{definition}\label{def:right/left_cotorsion_pair}
\mbox{}
\begin{enumerate}
\item
The pair $(\U,\V)$ of subcategories in $\C$ is called a \emph{right cotorsion pair}
if it satisfies the following conditions:
\begin{enumerate} 
\item
For any $X\in\C$, there exists a conflation
\begin{equation*}
V_X\to U_X\xto{p_X} X
\end{equation*}
where $p_X$ is a $\U$-deflation of $X$. In this case, we say that \textit{$X$ is resolved by $(\U,\V)$};
\item
Any $f\in\C(U,V)$ with $Y\in\U$ and $V\in\V$ factors through an object in $\U\cap\V$.
\end{enumerate}
For a right cotorsion pair $(\U,\V)$, we denote by $\udef$ the class of $\U$-deflations $p_X$ for all $X\in\C$. 
\item
The pair $(\S,\T)$ of subcategories in $\C$ is called a \emph{left cotorsion pair} if it satisfies the following conditions:
\begin{enumerate} 
\item
For any $X\in\C$, there exists a conflation
\begin{equation*}
X\xto{\iota_X}T^X\to S^X
\end{equation*}
where $\iota_X$ is a $\T$-inflation of $X$.
In this case, we also say that \textit{$X$ is resolved by $(\S,\T)$};
\item
Any $f\in\C(S,T)$ with $S\in\S$ and $T\in\T$ factors through an object in $\S\cap\T$.
\end{enumerate}
For a left cotorsion pair $(\S,\T)$, we denote by $\tinf$ the class of $\T$-inflations $\iota_X$ for all $X\in\C$. 
\end{enumerate}
\end{definition}

\begin{lemma}\label{lem:adjoint_pair}
For a right cotorsion pair $(\U,\V)$ with $\I:=\U\cap\V$,
an assignment $R:X\mapsto U_X$ gives rise to a right adjoint $R:\C/\I\to\U/\I$ of the canonical inclusion $\U/\I\hookrightarrow\C/\I$.
Dually, for a left cotorsion pair $(\S,\T)$ with $\I:=\S\cap\T$, an assignment $L:X\mapsto T^X$ gives rise to a left adjoint $L:\C/\I\to\T/\I$ of the inclusion $\T/\I\to\C/\I$.
\end{lemma}
\begin{proof}
Fix a conflation $V_X\to U_X\xto{p_X} X$ with $p_X$ being a $\U$-deflation for each $X\in\C$ and consider an assignment  $R:X\mapsto U_X$.
Due to the condition (b) in Definition \ref{def:right/left_cotorsion_pair}, for any morphism $f\in\C(X,Y)$, we have a unique morphism $Rf:RX\to RY$ which makes the following diagram commutative in $\C/\I$:
\[
\xymatrix@R=16pt{
V_X\ar[r]&RX\ar[r]^{p_X}\ar[d]_{Rf}&X\ar[d]^f\\
V_Y\ar[r]&RY\ar[r]^{p_Y}&Y
}
\]
Thus, it is easily checked that the assignment $R$ gives rise to a functor $\C/\I\to\U/\I$.
The $\U$-deflation $p_X$ induces a surjective morphism $\C/\I(U,RX)\xto{p_X\circ -}\C/\I(U,X)$ for any $U\in\U$.
Again, due to the condition (b), the above morphism $p_X\circ -$ is injective.
Hence the functor $R:\C/\I\to\U/\I$ is a right adjoint of the inclusion $\U/\I\hookrightarrow \C/\I$.
The assertion for a left cotorsion pair can be proved dually.
\end{proof}

Obviously a right/left cotorsion pair is a weaker notion of a cotorsion pair which are firstly defined in abelian categories \cite{Sal79} and recently in extriangulated categories \cite{NP19}.

\begin{definition}
A pair $(\U,\V)$ is called a \emph{cotorsion pair} if it satisfies the following conditions:
\begin{enumerate}
\item Both $\U$ and $\V$ are closed under direct summands, isomorphisms and extensions;
\item $\mathbb{E}(\U,\V)=0$;
\item $\cone(\V,\U)=\C=\cocone(\V,\U)$.
\end{enumerate}
\end{definition}

Note that, since we do not require $\mathbb{E}$-orthogonality, a left and right cotorsion pair does not form a cotorsion pair in general.
In fact, $(\C,\C)$ is a left and right cotorsion pair.
Moreover, we should remark that some authors use the terminology left/right cotorsion pair in a different sense \cite{BMPS19, Gil11}.

Let $(\S,\T)$ and $(\U,\V)$ be a left cotorsion pair and a right cotorsion pair, respectively,
and denote by $\P:=((\S,\T), (\U,\V))$ the pair of them.
We consider the following two conditions:
\begin{enumerate}
\item[(Hov1)] $\S\subseteq \U$ and $\T\supseteq \V$;
\item[(Hov2)] $\S\cap\T=\U\cap\V$.
\end{enumerate}

In the rest of this article, \textit{let $\P$ denote the above pair satisfying \textnormal{(Hov1)} and \textnormal{(Hov2)}.}
The following notions will be used in many places:

\begin{inparaenum}
 \item $\I:=\S\cap\V$; 
 \item $\Z:=\T\cap\U$; 
 \item $\mathbb{V}:=\udef\circ\tinf$,
\end{inparaenum}\\
where $\udef\circ\tinf$ denotes the class of morphisms $f$ which has a factorization $f=f_2\circ f_1$ with $f_1\in\tinf$ and $f_2\in\udef$.
Note that $\I=\U\cap\V$ holds.

\begin{definition}
If $\P$ satisfies the following two conditions, it is called a \textit{Hovey twin cotorsion pair} (\textit{HTCP}) introduced in \cite{NP19}:
\begin{enumerate}
\item[(Hov3)] Both $(\S,\T)$ and $(\U,\V)$ are cotorsion pairs;
\item[(Hov4)] $\cone(\V,\S)=\cocone(\V,\S)$.
\end{enumerate}
\end{definition}

We shall study basic properties of $\P$ in this section.
Denote by $\pi:\C\to\C/\I$ the additive quotient.
First, we consider two functors $Q_{LR}:=LR\pi$ and $Q_{RL}:=RL\pi :\C\to\C/\I$
by composing a right adjoint $R:\C/\I\to\U/\I$ and a left adjoint $L:\C/\I\to\T/\I$.
We also consider a Gabriel-Zisman localization $\mathsf{L}_\mathbb{V}:\C\to\C[\mathbb{V}^{-1}]$ with respect to the class $\mathbb{V}$.
The following proposition draws a comparison between $\Z/\I$ and $\C[\mathbb{V}^{-1}]$, which is instrumental to formulate our main theorem.

\begin{proposition}\label{prop:key}
We assume that both $\Im Q_{LR}$ and $\Im Q_{RL}$ are contained in $\Z/\I$.
Then, there uniquely exists a functor $\Phi :\Z/\I\to\C[\mathbb{V}^{-1}]$ which makes the following diagram commutative up to isomorphism:
$$
\xymatrix@C=36pt@R=12pt{
&{\Z/\I}\ar@{..>}[rd]^\Phi&\\
{\C}\ar[ru]^{Q_{LR}}\ar[rd]_{Q_{RL}}\ar[rr]|{\mathsf{L}_\mathbb{V}}&&{\C[\mathbb{V}^{-1}]}.\\
&{\Z/\I}\ar@{..>}[ru]_\Phi&
}$$
\end{proposition}

Before proving Proposition \ref{prop:key}, we check the next easy lemma.

\begin{lemma}\label{lem:for_key}
Let $\mathbb{V}_\Z$ be the class of morphisms in $\Z$ which consists of sections $f$ with $\Cok{f}\in\I$.
Then, a containment $\mathbb{V}_\Z\subseteq\mathbb{V}\cap\Mor\Z$ is true.
\end{lemma}
\begin{proof}
Let $X\xto{f} Y\to I$ be a splitting short exact sequence in $\Z$ with $I\in\I$.
The morphism $f$ is obviously a $\T$-inflation.
Thus $\mathbb{V}_\Z\subseteq \tinf$.
\end{proof}

\begin{proof}[Proof of Proposition \ref{prop:key}]
Let us recall from Example \ref{ex:univ_of_stable_cat} that the additive quotient $\pi: \Z\to\Z/\I$ can be regarded as a Gabriel-Zisman localization of $\Z$ with respect to $\mathbb{V}_\Z$.
First, we consider the functor $\mathsf{L}_\mathbb{V}|_\Z:\Z\to\C[\mathbb{V}^{-1}]$ restricted onto $\Z$.
Since $\mathbb{V}_\Z\subseteq\mathbb{V}$,
by the universality of $\pi$, 
we have a unique functor $\Phi:\Z/\I\to\C[\mathbb{V}^{-1}]$ with $\mathsf{L}_\mathbb{V}|_\Z\cong \Phi\circ\pi$.
Thus we have the following diagram commutative up to isomorphism:
\begin{equation}\label{diag:main_prop}
\xymatrix{
\Z\ar[d]_\pi\ar[r]^{\mathsf{inc}}&\C\ar[r]^{\mathsf{L}_\mathbb{V}\ \ \ }&\C[\mathbb{V}^{-1}]\\
\Z/\I\ar@{..>}[rru]_{\Phi}&&
}
\end{equation}
where $\mathsf{inc}$ means the canonical inclusion.

We shall show that the obtained functor $\Phi$ satisfies $\mathsf{L}_\mathbb{V}\cong\Phi\circ Q_{LR}$.
Let $X\in\C$.
Associated to the definition of $Q_{LR}$, we have a $\U$-deflation $p_X:RX\to X$ and $\T$-inflation $\iota_{RX}:RX\to LRX$ in $\C$, as depicted in the  diagram below:
\begin{equation}\label{diag:associated_conflations}
\xymatrix@R=16pt{
V_X\ar[r]&RX\ar[r]^{p_X}\ar[d]^{\iota_{RX}}&X\\
&LRX\ar[d]&\\
&S^{RX}&
}
\end{equation}
where the row and the column are conflations and $V_X\in\V$ and $S^{RX}\in\S$.
Since $\mathsf{L}_\mathbb{V}$ sends $\mathbb{V}$ to a class of isomorphisms, we have isomorphisms in $\C[\mathbb{V}^{-1}]$:
$$\mathsf{L}_\mathbb{V}(X)\xleftarrow{\mathsf{L}_\mathbb{V}(p_X)}\mathsf{L}_\mathbb{V}(RX)\xto{\mathsf{L}_\mathbb{V}(\iota_{RX})}\mathsf{L}_\mathbb{V}(LRX)$$
Since $LRX\in\Z$, the diagram (\ref{diag:main_prop}) shows a natural isomorphism $\mathsf{L}_\mathbb{V}(LRX)\cong \Phi\circ\pi(LRX)$.
By the assumption, we notice $\pi(LRX)=Q_{LR}(X)$ in $\Z/\I$ and obtain a desired natural isomorphism $\mathsf{L}_\mathbb{V}(X)\cong \Phi\circ Q_{LR}(X)$ in $X$.
By a similar argument, we have an isomorphism $\mathsf{L}_\mathbb{V}\cong\Phi\circ Q_{RL}$.

It remains to show the uniqueness of $\Phi$.
To this end, we assume that there exists a functor $\Phi'$ with $\mathsf{L}_\mathbb{V}\cong \Phi'\circ Q_{RL}$.
Since $Q_{RL}$ is an identity functor on $\Z/\I$, $\pi $ equalize both $\Phi$ and $\Phi'$, precisely, $\Phi'\circ\pi\cong\Phi\circ\pi\cong\mathsf{L}_\mathbb{V}|_\Z$.
However, by the universality of $\pi$, such functors are uniquely determined up to isomorphism.
Hence $\Phi\cong\Phi'$.
\end{proof}

We deduce the following result from Proposition \ref{prop:key}.

\begin{lemma}\label{lem:Q_sends_W_to_ISO}
If there exists a functorial isomorphism $\eta:Q_{LR}\xto{\sim} Q_{RL}$,
the functors $Q_{LR}\cong Q_{RL}$ send $\mathbb{V}$ to a class of isomorphisms in $\C/\I$.
Moreover, we have restricted functors $Q_{LR}\cong Q_{RL}:\C\to\Z/\I$.
This situation can be depicted below:
$$
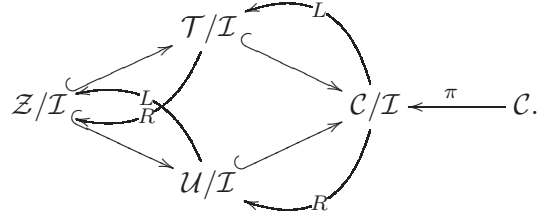
\xymatrix@C=36pt@R=12pt{
&{\T/\I}\ar@/^20pt/[ld]|{R}\ar@{^(->}[rd]&&\\
{\Z/\I}\ar@{^(->}[ru]\ar@{^(->}[rd]&&{\C/\I}
\ar@/^24pt/[ld]|{R}
\ar@/_24pt/[lu]|{L}&\C .\ar[l]_\pi\\
&{\U/\I}\ar@{^(->}[ru]\ar@/_20pt/[lu]|{L}&&
}$$
\end{lemma}
\begin{proof}
By definition, $Q_{LR}(\udef)$ and $Q_{RL}(\tinf)$ form classes of isomorphisms.
Thus the first assertion holds.
Since $\Im Q_{LR}\subseteq \T/\I$ and $\Im Q_{RL}\subseteq \U/\I$, 
the existence of $\eta$ forces $\Im Q_{RL}\in\Z/\I$.
\end{proof}

A commutativity of $R$ and $L$ is a key property to prove many assertions in this article.
So we introduce the following terminology.

\begin{definition}
Let $\P$ be a pair of a left cotorsion pair and a right cotorsion pair satisfying \textnormal{(Hov1)} and \textnormal{(Hov2)}.
If there exists a functorial isomorphism $\eta:Q_{LR}\xto{\sim}Q_{RL}$,
then the pair $\P$ is called a \textit{generalized Hovey twin cotorsion pair} (\textit{gHTCP}).
Morover, the associated localization $\mathsf{L}_\mathbb{V}:\C\to\C[\mathbb{V}^{-1}]$ is called the \textit{gHTCP localization} with respect to $\P$.
\end{definition}

The following is our main result.
\begin{theorem}\label{thm:RL=LR_implies_universality}
Let $\P$ be a pair satisfying \textnormal{(Hov1)} and \textnormal{(Hov2)}.
The following are equivalent:
\begin{itemize}
\item[\textnormal{(i)}] There exists a natural isomorphism $\eta:Q_{LR}\to Q_{RL}$;
\item[\textnormal{(ii)}] $R(\T/\I)\subseteq \Z/\I$ holds and the  functor $L:\U/\I\to\Z/\I$ sends $R\pi(\tinf)$ to isomorphisms;
\item[\textnormal{(iii)}] $L(\U/\I)\subseteq \Z/\I$ holds and the  functor $R:\T/\I\to\Z/\I$ sends $L\pi(\udef)$ to isomorphisms;
\item[\textnormal{(iv)}] Both $\Im Q_{LR}$ and $\Im Q_{RL}$ are contained in $\Z/\I$ and the functor $\Phi:\Z/\I\to\C[\mathbb{V}^{-1}]$ is an equivalence.
\end{itemize}
\end{theorem}
\begin{proof}
(i) $\Rightarrow$ (ii), (iii): These implications directly follow from Lemma \ref{lem:Q_sends_W_to_ISO}.

(ii) $\Rightarrow$ (i): For $X\in\C$, we consider a $\T$-inflation $X\xto{\iota_X} LX$, $\U$-deflations $p_X:RX\to X$ and $p_{LX}:RLX\to LX$.
Then we get a morphism $R\iota_X:RX\to RLX$ such that $\iota_X\circ p_X =p_{LX}\circ R\iota_X$.
In addition, taking $\T$-inflation starting from $RX$ yields the following commutative diagram in $\C$:
\begin{equation}\label{diag:eta}
\xymatrix@R=16pt{
X\ar[r]^{\iota_X}&LX\\
RX\ar[d]_{\iota_{RX}}\ar[r]^{R\iota_X}\ar[u]^{p_{X}}&RLX\ar[u]_{p_{LX}}\\
LRX&
}
\end{equation}
Since $RLX\in\Z$ and $\iota_{RX}$ is a left $\T$-approximation, we have a morphism $\eta_X:LRX\to RLX$ with $R\iota_X =\eta_X\circ\iota_{RX}$.
By the assumption, we have isomorphisms $L\pi(\iota_{RX})$ and $L\pi(R\iota_{X})$.
Hence we conclude that $L\pi (\eta_X)$ is also an isomorphism in $\Z/\I$.
Due to $Q_{LR}(X)\cong L\pi (LRX)$ and $Q_{RL}(X)\cong L\pi (RLX)$, the morphism $\eta_X$ gives rise to a desired natural isomorphism.

(iii) $\Rightarrow$ (i): It can be checked by the dual argument.

(iv) $\Rightarrow$ (ii), (iii): Since $\mathsf{L}_\mathbb{V}$ sends $\mathbb{V}$ to isomorphisms, the assertions (ii) and (iii) hold.

(i) $\Rightarrow$ (iv): We shall show that the functor $Q_{LR}:\C\to\Z/\I$ has the same universality as that of $\mathsf{L}_\mathbb{V}$.
Due to Lemma \ref{lem:Q_sends_W_to_ISO}, it follows that $Q_{LR}$ sends $\mathbb{V}$ to isomorphisms.
To check the universality of $Q_{LR}$, 
let $F:\C\to\D$ be a functor which sends $\mathbb{V}$ to a class of isomorphisms in a cateogry $\D$.
We shall construct a functor $F':\Z/\I\to\D$ such that $F\cong F'\circ Q_{LR}$ as follows:
For an object $X\in\mathsf{Ob}(\Z/\I)\subseteq\mathsf{Ob}(\C)$,
we set $F'(X):=F(X)$;
For a morphism $f:X\to Y$ in $\Z$, we set $F'(Q_{LR}(f)):=F(f)$.
This assignment can give rise to a desired functor $F'$.
To show the well-definedness of this assignment:
Let $f, g:X\to Y$ be morphisms in $\Z$ with $Q_{LR}(f)=Q_{LR}(g)$ in $\Z/\I$, namely, $f-g$ is factored as $X\xto{a}I\xto{b}Y$ with $I\in\I$.
Then we get $f=\begin{pmatrix}
g&b
\end{pmatrix}
\begin{pmatrix}
1\\
a
\end{pmatrix}
$.
Since the remaining argument to check the well-definedness is completely same as the proof of Example \ref{ex:univ_of_stable_cat}, we skip the details.

To verify the commutativity up to isomorphism, namely, the existence of an isomorphism $F\cong F'\circ Q_{LR}$,
we let $f:X\to Y$ be a morphism in $\C$ and consider the following commutative digram in $\C$:
$$
\xymatrix@R=16pt{
LRX\ar[d]^{LRf}&RX\ar[r]^{p_X}\ar[l]_{\iota_{RX}}\ar[d]^{Rf}&X\ar[d]^f\\
LRY&RY\ar[r]_{p_Y}\ar[l]^{\iota_{RY}}&Y
}
$$
where $R*\in\U$, $LR*\in\Z$, $p_*\in\udef$ and $\iota_*\in\tinf$ for $*=X,Y$.
Note that, due to an isomorphism $\eta:Q_{LR}\xto{\sim}Q_{RL}$, we get isomorphisms $Q_{LR}(\iota_*)$ and $Q_{LR}(p_*)$ for $*=X,Y$.
Since $LRf$ is a morphism in $\Z/\I$, by definition, we have an equality $F'\circ Q_{LR}(LRf)=F(LRf)$.
The functors $F$ and $Q_{LR}$ send the above horizontal arrows to  isomorphisms in $\D$ and $\Z/\I$, respectively.
Hence we have a desired natural isomorphism
\[
F(X)\xto{F'\circ Q_{LR}(p_X)F'\circ Q_{LR}(\iota_{RX})^{-1}F(\iota_{RX})F(p_X)^{-1}}F'\circ Q_{LR}(X)
\]
in $X\in\C$.

Finally, to show the uniqueness of a desired functor $F'$, we assume that there exists a functor $F''$ with $F\cong F''\circ Q_{LR}$.
Then, the functors restricts on $\Z$ to have isomorphisms $F|_\Z\cong F'\circ Q_{LR}|_\Z\cong F''\circ Q_{LR}|_\Z$.
Note that $F'\circ Q_{LR}|_\Z\cong F'\pi$ and $F''\circ Q_{LR}|_\Z\cong F''\pi$.
Since the ideal quotient $\Z\to \Z/\I$ is a localization with respect to $\mathbb{V}_\Z$ (which is a subclass of $\mathbb{V}$),
by the universality, we have $F'\cong F''$.
\end{proof}

Although the gHTCP localization $\mathsf{L}_\mathbb{V}$ does not admit a calculus of left and right fractions in general, the morphisms in $\C[\mathbb{V}^{-1}]$ admit the following nice descriptions.

\begin{corollary}
Suppose that a pair $\P$ forms a gHTCP.
For any morphism $\overline{\alpha}:\mathsf{L}_\mathbb{V}(X)\to\mathsf{L}_\mathbb{V}(Y)$ in $\C[\mathbb{V}^{-1}]$, there exist morphisms $s:RX\to X$, $t:Y\to LY$ in $\mathbb{V}$ and $\alpha:RX\to LY$ in $\C$ such that $\overline{\alpha}=\mathsf{L}_\mathbb{V}(t)^{-1}\circ\mathsf{L}_\mathbb{V}(\alpha)\circ\mathsf{L}_\mathbb{V}(s)^{-1}$.
\end{corollary}
\begin{proof}
Due to the equivalence $\Phi:\Z/\I\xto{\sim}\C[\mathbb{V}^{-1}]$,
we have a morphism $\overline{f}:Q_{LR}(X)\to Q_{LR}(Y)$ in $\Z/\I$ such that $\Phi(\overline{f})=\overline{\alpha}$.
Regarding $Q_{LR}(X), Q_{RL}(Y)\in\Z$, we get a morphism $f:Q_{LR}(X)\to Q_{RL}(Y)$ with $\pi(f)=\overline{f}$.
For $X\in\C$, there exists a diagram $X\xleftarrow{p_X}RX\xto{\iota_{RX}}Q_{LR}X$ with $p_X\in\udef$ and $\iota_{RX}\in\tinf$.
Similarly, for $Y\in\C$, we have a diagram $Y\xto{\iota_Y}LY\xleftarrow{p_{LY}}Q_{RL}Y$ with $p_{LY}\in\udef$ and $\iota_{Y}\in\tinf$.
Our situation can be depicted as:
$$
\xymatrix@R=16pt{
V_X\ar[d]&&\\
RX\ar[d]^{p_X}\ar[r]^{\iota_{RX}}&Q_{LR}X\ar[dd]^f\ar[r]&S^{RX}\\
X&&Y\ar[d]^{\iota_Y}\\
V_{LX}\ar[r]&Q_{RL}Y\ar[r]^{p_{LY}}&LY\ar[d]\\
&&S^Y
}
$$
Setting $\alpha:=p_{LY}\circ f\circ \iota_{RX}$,
we have obtained $\overline{\alpha}=\mathsf{L}_\mathbb{V}(\iota_Y)^{-1}\circ\mathsf{L}_\mathbb{V}(\alpha)\circ\mathsf{L}_\mathbb{V}(p_X)^{-1}$.
\end{proof}

\subsection{Comparison between gHTCP and HTCP}\label{ssec:comparison}
Throughout this subsection, we always suppose that the extriangulated category $\C$ satisfies the following condition (WIC):
\begin{condition}[WIC]
For a given extriangualted category $\C$, we consider the following conditions.
\begin{enumerate}
\item For a composed morphism $g\circ f$ in $\C$, if $g\circ f$ is an inflation, then so is $f$.
\item For a composed morphism $g\circ f$ in $\C$, if $g\circ f$ is an deflation, then so is $g$.
\end{enumerate}
\end{condition}

To bridge the gap between gHTCP and HTCP, we consider some additional conditions on the gHTCP.
Keeping in mind Nakaoka-Palu's correspondence theorem between HTCP's and admissible model structures on $\C$ \cite[Section 5]{NP19},
for a pair $\P$ of a left cotorsion pair and a right one,
we consider the following classes of morphisms:
\begin{itemize}
\item $\wfib:=$\,the class of deflations $f$ with $\cocone(f)\in\V$;
\item $\wcof:=$\,the class of inflations $f$ with $\cone(f)\in\S$;
\item $\mathbb{W}:=\wfib\circ\wcof$.
\end{itemize}
where $\wfib\circ\wcof$ denotes the class of morphisms $f$ which has a factorization $f=f_2\circ f_1$ with $f_1\in\wcof$ and $f_2\in\wfib$.
Obviously, we have $\udef\subseteq \wfib$, $\tinf\subseteq \wcof$ and $\mathbb{V}\subseteq\mathbb{W}$.
Recall the definition of HTCP localization.

\begin{theorem}\cite[Cor. 5.25]{NP19}\label{thm:Nakaoka-Palu's_HTCP}
Let $\P$ be an HTCP.
Then, the Gabriel-Zisman localization $\mathsf{L}_\mathbb{W}:\C\to\C[\mathbb{W}^{-1}]$ induces an equivalence $\Psi:\Z/\I\to\C[\mathbb{W}^{-1}]$, unique up to isomorphism, which is depicted as follows:
$$
\xymatrix{
\Z\ar[d]_\pi\ar@{^(->}[r]^{\mathsf{inc}}&\C\ar[r]^{\mathsf{L}_\mathbb{W}\ \ }&\C[\mathbb{W}^{-1}]\\
\Z/\I\ar@{..>}[rur]_\Psi&&
}
$$
We call this localization the \textnormal{HTCP localization} of $\C$ with respect to $\P$.
\end{theorem}

The aim of this subsection is to show that gHTCP is a generalization of HTCP in the following sense:

\begin{theorem}\label{thm:HTCP_is_gHTCP}
Let $\P=((\S,\T),(\U,\V))$ be an HTCP.
Then $\P$ is also a gHTCP.
Moreover, the HTCP localization $\C[\mathbb{W}^{-1}]$ is equivalent to the gHTCP localization $\C[\mathbb{V}^{-1}]$.
\end{theorem}

To prove the theorem, we shall use the following easy lemmas.
\begin{lemma}\label{lem:misc1}
Let $(\U,\V)$ be a right cotorsion pair in an extriangulated category $\C$.
Assume that  $\V$ is extension-closed and $\mathbb{E}(\U,\V)=0$.
Then $R\pi:\C\to\U/\I$ sends $\wfib$ to a class of isomorphisms.
\end{lemma}
\begin{proof}
Take a morphism $f\in\wfib$ together with a conflation $V\xto{}X\xto{f}Y$ with $V\in\V$.
We resolve $X$ by $(\U,\V)$ to get a conflation $V_X\to U_X\xto{p_X} X$ with $V_X\in\V$ and $U_X\in\U$.
Thanks to the following commutative diagram and extension-closedness of $\V$,  we have that the composed morphism $f\circ p_X$ is in $\wfib$:
$$
\xymatrix@R=16pt{
V_X\ar[d]\ar@{=}[r]&V_X\ar[d]&\\
V'\ar[r]\ar[d]&U_X\ar[r]^{f\circ p_X}\ar[d]_{p_X}&Y\ar@{=}[d]\\
V\ar[r]&X\ar[r]^f&Y
}
$$
Since $\mathbb{E}(\U,\V)=0$, we get $f\circ p_X\in\udef$.
Hence we conclude $RX=U_X\cong RY$ in $\U/\I$.
\end{proof}

\begin{lemma}\label{lem:contained_in_Z}
Let $\P$ be a pair of left cotorsion pair and a right cotorsion pair satisfying \textnormal{(Hov1)} and \textnormal{(Hov2)}.
If both $\T$ and $\U$ are extension-closed, then the images of $Q_{LR}$ and $Q_{RL}$ are contained in $\Z/\I$.
\end{lemma}
\begin{proof}
Let $X\in\C$ and consider the following commutative diagram associated to the definition of $Q_{LR}$:
$$
\xymatrix@R=16pt{
V_X\ar[r]&RX\ar[d]^{\iota_{RX}}\ar[r]^{p_{X}}&X\\
&LRX\ar[d]&\\
&S^{RX}&
}
$$
which is same as (\ref{diag:associated_conflations}).
Since $\U$ is extension-closed, $RX, S^{RX}\in\U$ forces $LRX\in\Z$.
This says $\Im Q_{LR}\subseteq \Z/\I$.

The assertion for $Q_{RL}$ can be checked dually.
\end{proof}

The following proposition shows the latter statement in Theorem \ref{thm:HTCP_is_gHTCP}.

\begin{proposition}\label{prop:add}
Let $\P$ be an HTCP.
Then, there exists an equivalence $F:\C[\mathbb{V}^{-1}]\xto{\sim}\C[\mathbb{W}^{-1}]$.
\end{proposition}
\begin{proof}
Since $\mathbb{V}\subseteq\mathbb{W}$, there exists a natural functor $F:\C[\mathbb{V}^{-1}]\to\C[\mathbb{W}^{-1}]$ such that $\mathsf{L}_\mathbb{W}\cong F\circ\mathsf{L}_\mathbb{V}$.
Denote by $\widetilde{\mathbb{V}}$ the class of morphisms $f\in\Mor\C$ such that $\mathsf{L}_\mathbb{V}(f)$ is an isomorphism.
We also use the symbol $\widetilde{\mathbb{W}}$ in an obvious meaning.
A basic property of the model category shows $\widetilde{\mathbb{W}}=\mathbb{W}$ (see \cite[Remark 7.8]{NP19}).
By Lemma \ref{lem:misc1} and the dual, we conclude $\mathbb{W}\subseteq\widetilde{\mathbb{V}}=\widetilde{\mathbb{W}}$ which forces $F$ to be an equivalence.
\end{proof}
\begin{proof}[Proof of Theorem \ref{thm:HTCP_is_gHTCP}]
By Proposition \ref{prop:key} and Lemma \ref{lem:contained_in_Z}, there exists a functor $\Phi:\Z/\I\to\C[\mathbb{V}^{-1}]$.
We only have to show that $\Phi$ is an equivalence.
Thanks to Theorem \ref{thm:Nakaoka-Palu's_HTCP},
we have the following diagram commutative up to isomorphism:
$$
\xymatrix{
\Z\ar[d]_\pi\ar@{^(->}[r]^{\mathsf{inc}}&\C\ar[d]^{\mathsf{L}_\mathbb{V}}\ar[r]^{\mathsf{L}_\mathbb{W}\ \ }&\C[\mathbb{W}^{-1}]\\
\Z/\I\ar[r]^\Phi\ar@/_3.5pc/[rru]_\Psi&\C[\mathbb{V}^{-1}]\ar[ru]_F&
}
$$
The commutativity $\Psi\cong F\circ\Phi$ follows from the uniqueness of $\Psi$.
Since $F$ and $\Psi$ are equivalences, so is $\Phi$.
This finishes the proof.
\end{proof}

\begin{remark}
We remark that the existence of the above isomorphism $\eta$ follows from the model structure corresponding to the HTCP via the general theory of model category, e.g. \cite[Ch. 1.1]{DS95}, \cite[Thm. 1.2.10]{Hov99}.
\end{remark}

\subsection{Examples: recollement of abelian/triangulated categories}\label{ssec:example_of_gHTCP}
In this section, as applications of Theorem \ref{thm:RL=LR_implies_universality},
we interpret some important phenomena via gHTCP localizations.
\subsubsection{Additive quotient}
Let $\C$ be an extriangulated category and $\D$ its full subcategory which is closed under isomorphisms and direct summands.
Then, $\P:=((\D,\C),(\C,\D))$ obviously forms a gHTCP.
The following is straightforward.

\begin{corollary}
The following hold for the above $\P$.
\begin{itemize}
\item[(1)] $\I=\D$ and $\Z=\C$.
\item[(2)] The class $\udef$ consists of retractions with the cocones belonging to $\D$.
\item[(3)] The class $\tinf$ consisits of sections with the cones belonging to $\D$.
\item[(4)] The functors $L\pi, R\pi:\C/\D\to\C/\D$ are isomorphic to the identity functors.
\end{itemize}
In particular, we have an equivalence $\Phi:\C/\D\xto{\sim} \C[\mathbb{V}^{-1}]$.
\end{corollary}

This equivalence $\Phi$ is nothing other than the equivalence $\C/\I\xto{\sim}\C[\mathbb{S}^{-1}]$ appearing in Example \ref{ex:univ_of_stable_cat}.
Note that the above construction of $\C[\mathbb{V}^{-1}]$ does not depend on extriangulated structures on $\C$.
Since an additive category always has a splitting exact structure, any additive quotient can be considered as a gHTCP localization.

\subsubsection{Recollement of triangulated categories}
We shall explain that a recollement of triangulated categories gives an example of HTCP localizations.
More generally, we shall investigate the following Iyama-Yang's realization of Verdier quotients as subfactors via HTCP localizations (see also \cite{Li16}).
We introduce the following notions.

\begin{definition}
Let $\C$ be an extriangulated category and $\N$ its full subcategory.
We define the full subcategories
\begin{itemize}
\item $\N^{\perp_1}:=\{X\in\C\mid \mathbb{E}(\N, X)=0\}$;
\item $\N^{\perp}:=\{X\in\C\mid \mathbb{E}(\N, X)=0=\C(\N, X)\}$.
\end{itemize}
The subcategories ${}^{\perp_1}\N$ and ${}^{\perp}\N$ are defined dually.
\end{definition}

\begin{corollary}\label{cor:IYa17}
Let $\C =(\C, [1], \Delta)$ be a triangulated category and $\N$ its thick subcategory
and denote by $\C_\N$ the Verdier quotient of $\C$ by $\N$.
We suppose that $\N$ has a cotorsion pair $(\S,\V)$; $\C$ has cotorsion pairs $(\S, \S^{\perp_1}), ({^{\perp_1}\V},\V)$.
Then, $\P:=((\S, \S^{\perp_1}),({^{\perp_1}\V},\V))$ forms an HTCP.
Moreover, the canonical functor $\Z\subseteq\C\to\C_\N$ induces an equivalence $\Z/\I\simeq \C_\N$.
Here we use the symbols $\S^{\perp_1}$ and ${^{\perp_1}\V}$ regarding $\S$ and $\V$ as subcategories in $\C$.
\end{corollary}
\begin{proof}
Since $\P$ forms an HTCP, we have an equivalence $\Z/\I\simeq \C[\mathbb{W}^{-1}]$.
Let $\mathbb{S}$ be the multiplicative system associated to the thick subcategory $\N$, namely, $\mathbb{S}:=\{f\in\Mor(\C)\mid \cone(f)\in\N\}$.
We identify $\C_\N$ with $\C[\mathbb{S}^{-1}]$.
It suffices to show $\mathbb{W}=\mathbb{S}$.
By definition, we get $\mathbb{W}\subseteq \mathbb{S}$.
To show the converse, consider a triangle $X\xto{f}Y\to N\to X[1]$ with $N\in\N$.
Resolving $N$ by the cotorsion pair $(\S,\V)$, we obtain the following commutative diagram:
$$
\xymatrix@R=16pt{
&V_N\ar[d]\ar@{=}[r]&V_N\ar[d]\\
X\ar[r]^{f_1}\ar@{=}[d]&Y'\ar@{}[rd]|{\textnormal{(Pb)}}\ar[d]_{f_2}\ar[r]&S_N\ar[d]\\
X\ar[r]^f&Y\ar[r]&N
}
$$
where all rows and columns are triangles and $S_N\in\S, V_N\in\V$.
Thus we get a factorization $f=f_2\circ f_1$ with $f_1\in\wcof$ and $f_2\in\wfib$.
Hence $f\in\mathbb{W}$.
\end{proof}

\begin{remark}
The equivalence $\Z/\I\simeq\C_\N$ is obtained in \cite[Thm. 1.1]{IYa17} with a more explicit construction.
\end{remark}

The following result can be regarded as a special case of Corollary \ref{cor:IYa17}.

\begin{corollary}\cite[Cor. 6.20]{Nak18}\label{cor:Nak18}
Let $(\N, \C, \C_{\N})$ be a recollement of triangulated categories and $(\S,\V)$ a cotorsion pair in $\N$.
Then, the pair $\P:=((\S,\S^{\perp_{1}}),({^{\perp_1}\V}, \V))$ forms an HTCP in $\C$.
Moreover, the canonical functor $\Z\subseteq\C\to\C_\N$ induces an equivalence $\Z/\I\simeq \C_\N$.
\end{corollary}

Let us remark that, in Corollary \ref{cor:Nak18}, the thick subcategory $\N$ always has a cotorsion pair, for example, by setting $(\S,\V):=(\N,0)$.

\subsubsection{Recollement of abelian categories}
It is pointed out in \cite[Rem. 5.27]{NP19} that a recollement of abelian categories can be rarely regarded as an HTCP localization.
However, it is still a gHTCP localization under a certain assumption.
Recall the definition of recollement and its needed properties.

\begin{definition-proposition}\label{def:recollement}
Let $\C, \N$ and $\C_\N$ be abelian categories.
A \textnormal{recollement} of abelian categories is a diagram of functors between abelian categories of the form
$$\xymatrix@C=1.2cm{\N\ar[r]|-{\mathsf{i}}
&\C\ar[r]|-{\mathsf{e}}\ar@/^1.2pc/[l]^-{\mathsf{p}}\ar_-{\mathsf{q}}@/_1.2pc/[l]
&\C_\N \ar@/^1.2pc/[l]^{\mathsf{r}}\ar@/_1.2pc/[l]_{\mathsf{l}}}$$
with the conditions:
\begin{itemize}
\item $(\mathsf{q},\mathsf{i},\mathsf{p})$ and $(\mathsf{l},\mathsf{e},\mathsf{r})$ form adjoint triples;
\item The functors $\mathsf{i}, \mathsf{l}$ and $\mathsf{r}$ are fully faithful;
\item $\Im \mathsf{i}=\Ker \mathsf{e}$.
\end{itemize}
We denote the recollement by $(\N,\C,\C_\N)$ for short. 
In this case, 
\begin{enumerate}
\item[\textnormal{(1)}] $\Im \mathsf{r}=\N^{\perp}$ and $\Im \mathsf{l}={}^{\perp}\N$ hold;
\item[\textnormal{(2)}] for any $X\in\C$, there exist the following exact sequences
\begin{eqnarray}
&&0\to N\to \mathsf{le}(X)\to X\to \mathsf{iq}(X)\to 0\label{left_defining}\\
&&0\to \mathsf{ip}(X)\to X\to \mathsf{re}(X)\to N'\to 0\label{right_defining}
\end{eqnarray}
with $N,N'\in\N$.
\end{enumerate}
\end{definition-proposition}

We refer to \cite[Section. 2]{Psa14} for details.
A recollement can be considered as a special case of Serre quotients in the following sense.

\begin{proposition}\cite[Rem. 2.3]{Psa14}
Consider the Serre quotient $\C_\N$ of an abelian category $\C$ with respect to its Serre subcategory $\N$.
If the quotient functor $\mathsf{e}:\C\to\C_\N$ admits a right adjoint $\mathsf{r}$ and a left adjoint $\mathsf{l}$, then the inclusion $\mathsf{i}:\N\to \C$ also admits a right adjoint $\mathsf{p}$ and a left adjoint $\mathsf{q}$.
Moreover, these six functors form a recollement.
\end{proposition}

The following is an abelian version of Corollary \ref{cor:Nak18}.

\begin{corollary}\label{cor:abelian_recollement}
Let $(\N, \C, \C_{\N})$ be a recollement of abelian categories.
Assume that $\C$ has enough projectives and enough injectives, and $\N$ has a cotorsion pair $(\S,\V)$ with $\Ext^2_\C(\S,\V)=0$.
Then, the pair $((\S,\S^{\perp_{1}}),({^{\perp_1}\V}, \V))$ forms a gHTCP in $\C$.
Moreover, the canonical functor $\Z\subseteq\C\to\C_\N$ induces an equivalence $\Z/\I\simeq \C_\N$.
Here we use the symbols $\S^{\perp_1}$ and ${^{\perp_1}\V}$ regarding $\S$ and $\V$ as subcategories in $\C$.
\end{corollary}

To show Corollary \ref{cor:abelian_recollement}, we include more investigations on an extriangulated category with a gHTCP.

\begin{proposition}\label{prop:Loc_sends_W'_to_iso}
Let $\C$ be an extriangulated category with a gHTCP $\P$.
Assume that both $\S$ and $\V$ are extension-closed and $\mathbb{E}(\S,\T)=\mathbb{E}(\U,\V)=0$.
Then the functor $\mathsf{L}_\mathbb{V}:\C\to\C[\mathbb{V}^{-1}]$ sends $\mathbb{W}$ to a class of isomorphisms.
Moreover, there exists an equivalence $\Psi:\C[\mathbb{V}^{-1}]\xto{\sim}\C[\mathbb{W}^{-1}]$, uniquely up to isomorphism, such that $\mathsf{L}_\mathbb{W}\cong\Psi\circ\mathsf{L}_\mathbb{V}$.
\end{proposition}
\begin{proof}
By Lemma \ref{lem:misc1} and the dual, we know that $Q_{LR}(\wfib)$ and $Q_{RL}(\wcof)$ form classes of isomorphisms in $\Z/\I$.
Since $\P$ is a gHTCP, the functor $Q_{LR}\cong Q_{RL}\cong \mathsf{L}_\mathbb{V}$ sends $\mathbb{W}$ to a class of isomorphisms.
Since $\mathbb{V}\subseteq\mathbb{W}$, the functor $\mathsf{L}_\mathbb{W}:\C\to\C[\mathbb{W}^{-1}]$ sends $\mathbb{V}$ to isomorphisms. The second assertion follows from the universality.
\end{proof}

We are in position to prove Corollary \ref{cor:abelian_recollement}.

\begin{proof}

(A) The first step shows that the pair $((\S,\S^{\perp_{1}}),({^{\perp_1}\V}, \V))$ is a gHTCP.

(1) We shall show that $({^{\perp_1}\V},\V)$ is a right cotorsion pair.
Let $X$ be an object in $\C$ and consider the first syzygy of $X$, namely, an exact sequence
\begin{equation}\label{first_syzygy}
0\to \Omega X\xto{a} P(X)\to X\to 0
\end{equation}
in $\C$ with $P(X)\in\proj(\C)$.
Next, we consider an exact sequence (\ref{left_defining}) of $\Omega X$ in Proposition \ref{def:recollement}
\begin{equation}\label{left-defining}
0\to N\to \mathsf{le}(\Omega X)\to \Omega X\xto{f} \mathsf{iq}(\Omega X)\to 0
\end{equation}
with $N, \mathsf{iq}(\Omega X)\in\N$ and $\mathsf{le}(\Omega X)\in{^{\perp}\N}$.
Resolve $\mathsf{iq}(\Omega X)$ by the cotorsion pair $(\S,\V)$, and get an exact sequence $0\to \mathsf{iq}(\Omega X)\xto{g} V\to S\to 0$ with $V\in\V$ and $S\in\S$.
Put $h:=g\circ f:\Omega X\to V$.
Since $f$ is surjective and $g$ is injective, we have $\Cok h\cong S$ and $\Ker h = \Ker f$.
By taking a pushout of $a:\Omega X\to P(X)$ along $h$, we obtain the following commutative diagram:
$$
\xymatrix@R=16pt{
&0\ar[d]&0\ar[d]&&\\
&\Ker h\ar[d]\ar@{=}[r]&\Ker h\ar[d]&&\\
0\ar[r]&\Omega X\ar@{}[rd]|{\textnormal{(Po)}}\ar[d]_h\ar[r]^a&P(X)\ar[d]^{h'}\ar[r]&X\ar@{=}[d]\ar[r]&0\\
0\ar[r]&V\ar[d]\ar[r]&X'\ar[d]^b\ar[r]^c&X\ar[r]&0\\
&S\ar[d]\ar@{=}[r]&S\ar[d]&&\\
&0&0&&
}
$$
in which all rows and columns are exact.
The second row is a desired one.
It suffices to show $X'\in {^{\perp_1}\V}$.
Since $\Ker h$ is a factor of $\mathsf{le}(\Omega X)\in {}^{\perp}\N$,
we get $\Ker h\in{}^{\perp_0}\N$.
We deduce from the second column in the above diagram an exact sequence $0\to \Ker h\to P(X)\to \Ker b\to 0$.
The fact $\Ker h\in{}^{\perp_0}\V$ and $P(X)\in{}^{\perp_1}\V$ forces $\Ker b\in{}^{\perp_1}\V$.
The exact sequence $0\to \Ker b\to X'\to S\to 0$ shows $X'\in{}^{\perp_1}\V$.
Therefore, the morphism $c:X'\to X$ is a ${}^{\perp_1}\V$-deflation associated to the right cotorsion pair $({}^{\perp_1}\V,\V)$.
By the dual argument, we can verify that $(\S,\S^{\perp_1})$ forms a left cotorsion pair.

(2) The condition (Hov1) is obvious. 

(3) The condition (Hov2) follows from $\S\cap\S^{\perp_1}=\S\cap\V={}^{\perp_1}\V\cap\V$.

(4) We shall show that $\Im Q_{LR}$ is contained in $\Z/\I$.
Take an object $X\in\C$.
Following the definition of $Q_{LR}$, we consider the following commutative diagram
\begin{equation}\label{To_show_LR=RL}
\xymatrix@R=16pt{
&&0\ar[d]&0\ar[d]&\\
0\ar[r]&V_X\ar@{=}[d]\ar[r]&RX\ar[d]_{\iota_{RX}}\ar[r]^{p_X}&X\ar[d]^\iota\ar[r]&0\\
0\ar[r]&V_X\ar[r]&LRX\ar[d]\ar[r]^p&X'\ar[d]\ar[r]&0\\
&&S^{RX}\ar[d]\ar@{=}[r]&S^{RX}\ar[d]&\\
&&0&0&
}
\end{equation}
where all rows and columns are exact,
$p_X$ is a $\U$-deflation of $X$ and $\iota_{RX}$ is a $\T$-inflation of $RX$.
Since $RX, S^{RX}\in{}^{\perp_1}\V$, the second column shows $LRX\in {}^{\perp_1}\V\cap \S^{\perp_1}=\Z$.
Dually, we get $\Im Q_{LR}\subseteq \Z/\I$.

(5) We shall show that, if $\Ext^2_\C(\S,\V)=0$, there exists an isomorphism $\eta:Q_{LR}\xto{\sim}Q_{RL}$.
It is enough to show $X'\in\S^{\perp_1}$ in the diagram (\ref{To_show_LR=RL}).
Applying $\C(\S,-)$ to the middle row in (\ref{To_show_LR=RL}),
we have an exact sequence
$\Ext^1_\C(\S,LRX)\to\Ext^1_\C(\S,X')\to\Ext^2_\C(\S,V_X)$ in which the both sides are zero.
Thus we have $X'\in\S^{\perp_1}$.
Hence $\iota$ is a $\T$-inflation of $X$ and $p$ is a $\U$-deflation of $X'$.
We conclude that the pair $((\S,\S^{\perp_1}),({}^{\perp_1}\V,\V))$ forms a gHTCP.

(B) We shall show that there exists an equivalence $\Z/\I\simeq\C_{\N}$.
Let $\mathbb{S}$ be the class of morphisms whose kernels and cokernels are contained in $\N$.
An equality $\mathbb{W}=\mathbb{S}$ can be checked by the same method as in the proof of Corollary \ref{cor:IYa17}.
Therefore, we identify the Serre localization with $\mathsf{L}_\mathbb{W}:\C\to\C[\mathbb{W}^{-1}]$.
By Propsotion \ref{prop:Loc_sends_W'_to_iso}, we obtain a desired equivalence $\C[\mathbb{V}^{-1}]\simeq\C[\mathbb{W}^{-1}]$.
\end{proof}

Let us add a short argument on an extriangulated structure on $\Z/\I$ in Corollary \ref{cor:abelian_recollement}.
Note that, if both $\T$ and $\U$ are extension-closed, so is $\Z$.
Furthermore, as stated below, the subfactor $\Z/\I$ also has a natural extriagulated structure.

\begin{lemma}\label{lem:extri_str_on_Z/I}
Let $\C$ be an extriangulated category with a gHTCP $\P$.
Suppose that $\T$ and $\U$ are extension-closed and $\mathbb{E}(\S,\T)=0=\mathbb{E}(\U,\V)$ holds.
Then, $\Z$ is extension-closed and $\mathbb{E}(\I,\Z)=0=\mathbb{E}(\Z,\I)$ holds.
Moreover, $\Z/\I$ is an extriangulated category.
\end{lemma}
\begin{proof}
Since $\T$ and $\U$ are extension-closed, so is $\Z$.
By Proposition \ref{prop:closed_under_operation}(1), $\Z$ is an extriangulated cateogry.
The equations directly follow from $\mathbb{E}(\S,\T)=0=\mathbb{E}(\U,\V)$.
Thus, $\I$ is a class of projective-injective objects in $\Z$.
Due to Propositon \ref{prop:closed_under_operation}(2), we have the latter assertion.
\end{proof}

We focus the equivalence $\C_\N\xto{\sim}\Z/\I$ in Corollary \ref{cor:abelian_recollement}.
It is natural to compare the abelian exact structure on $\C_\N$ and the above extriangulated structure on $\Z/\I$.
The following example indicates that they do not coincide.

\begin{example}
We consider the Auslander algebra $B$ of the path algebra $A$ determined by the quiver $[\bullet\to \bullet\to \bullet]$.
The algebra $B$ can be considered as the one defined by the quiver with relations:
$$
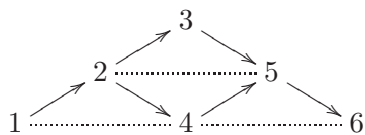
\xymatrix@!C=8pt@R=6pt{
&&3\ar[rd]&&\\
&2\ar@{.}[rr]\ar[ru]\ar[rd]&&5\ar[rd]&\\
1\ar[ru]\ar@{.}[rr]&&4\ar@{.}[rr]\ar[ru]&&6
}
$$
where the dotted line stands for the natural mesh relation.
The Auslander-Reiten quiver of the category $\C:=\mod B$ of finite dimensional modules is the following:
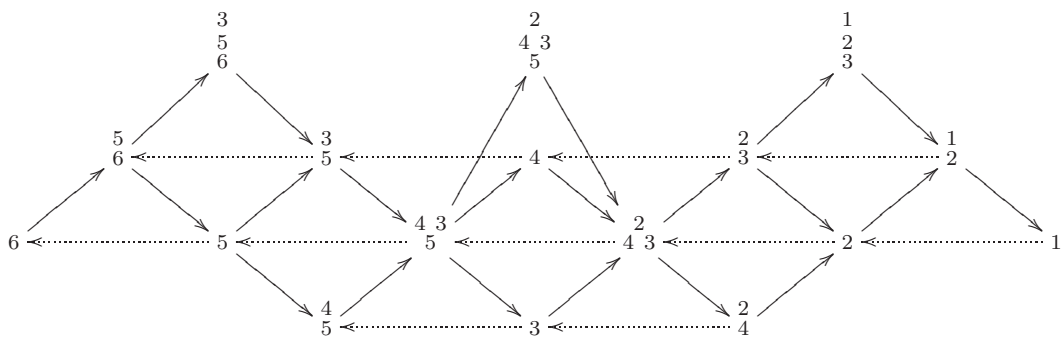
\begin{equation*}
\xymatrix@!C=15pt@R=12pt{
&&{\overset{3}{\overset{5}{\scriptstyle 6}}}\ar[rd]&&&\overset{2}{\overset{4\ 3}{\scriptstyle 5}}\ar[rdd]&&&\overset{1}{\overset{2}{\scriptstyle 3}}\ar[rd]&&\\
&\overset{5}{\scriptstyle 6}\ar[ru]\ar[rd]&&\overset{3}{\scriptstyle 5}\ar@{..>}[ll]\ar[rd]&&{\scriptstyle 4}\ar@{..>}[ll]\ar[rd]&&\overset{2}{\scriptstyle 3}\ar@{..>}[ll]\ar[ru]\ar[rd]&&\overset{1}{\scriptstyle 2}\ar@{..>}[ll]\ar[rd]&\\
{\scriptstyle 6}\ar[ru]&&{\scriptstyle 5}\ar@{..>}[ll]\ar[ru]\ar[rd]&&\overset{4\ 3}{\scriptstyle 5}\ar@{..>}[ll]\ar[ru]\ar[rd]\ar[ruu]&&\overset{2}{\scriptstyle 4\ 3}\ar@{..>}[ll]\ar[ru]\ar[rd]&&{\scriptstyle 2}\ar@{..>}[ll]\ar[ru]&&{\scriptstyle 1}\ar@{..>}[ll]\\
&&&\overset{4}{\scriptstyle 5}\ar[ru]&&{\scriptstyle 3}\ar@{..>}[ll]\ar[ru]&&\overset{2}{\scriptstyle 4}\ar@{..>}[ll]\ar[ru]&&&
}
\end{equation*}
where the dotted arrows denote the Auslander-Reiten translation.
We denote by ``$\circ$'' in a quiver the objects belonging to a subcategory and by ``$\cdot$'' the objects do not.
We consider the (injectively) stable Auslander algebra $\overline{B}$ and the associated inclusion $\N:=\mod\overline{B}\hookrightarrow\mod B$.
Then, it is well-known that $\mod\overline{B}$ is a Serre subcategory in $\mod B$:
\begin{equation*}
\xymatrix@!C=6pt@R=6pt{
&&\cdot&&&\cdot&&&\cdot&&\\
&\circ&&\cdot&&\circ&&\cdot&&\cdot&\\
\circ&&\circ&&\cdot&&\cdot&&\cdot&&\cdot\\
&&&\circ&&\cdot&&\cdot&&&
}
\end{equation*}
There exists a cotorsion pair $(\S, \V):=(\mod\overline{B},\mathsf{inj}\overline{B})$ in $\mod\overline{B}$ with $\Ext^2_B(\S,\V)=0$, where $\mathsf{inj}\overline{B}:=\inj(\mod\overline{B})$.
In fact, we have $\add({\scriptstyle 4}\oplus\overset{4}{\scriptstyle 5}\oplus\overset{5}{\scriptstyle 6})=\mathsf{inj}\overline{B}$ and the following calculations:
\begin{eqnarray*}
\Ext^2_\C(\S,{\scriptstyle 4})&\cong&\Ext^1_\C(\S,{\scriptstyle 2})\ =\ 0\\
\Ext^2_\C(\S,\overset{4}{\scriptstyle 5})&\cong&\Ext^1_\C(\S,\overset{2}{\scriptstyle 3})\ =\ 0\\
\Ext^2_\C(\S,\overset{5}{\scriptstyle 6})&\cong&\Ext^1_\C(\S,{\scriptstyle 3})\ =\ 0
\end{eqnarray*}
Corollary \ref{cor:abelian_recollement} guarantees that the pair
$$\P:=((\S,\T),(\U,\V)):=((\mod\overline{B}, (\mod\overline{B})^{\perp_1}),({}^{\perp_1}(\mathsf{inj}\overline{B}) ,\mathsf{inj}\overline{B}))$$
forms a gHTCP in $\mod B$.
We calculate to get that the associated subcategory $\Z=\T\cap\U$ cosisting of $\circ$ and $\bullet$:
\begin{equation*}
\xymatrix@!C=6pt@R=6pt{
&&\circ&&&\circ&&&\circ&&\\
&\bullet&&\cdot&&\bullet&&\cdot&&\circ&\\
\cdot&&\cdot&&\cdot&&\cdot&&\cdot&&\circ\\
&&&\bullet&&\cdot&&\circ&&&
}
\end{equation*}
The symbols $\bullet$ denotes the subcategory $\I$.
Since $\Ext^1_B(\Z,\Z)=0$,
the natural extriangulated structure of $\Z/\I$ is splitting, more precisely it is a splitting exact structure.
On the other hand, known as Auslander's formula \cite{Aus66, Len98}, the Serre localization $\C_\N=\frac{\mod B}{\mod \overline{B}}$ is equivalent to $\mod A$.
Hence, the equivalence $\C_\N\xto{\sim}\Z/\I$ obtained in Corollary \ref{cor:abelian_recollement} is not necessarily exact.
\end{example}

\section{Aspects of the heart construction}\label{sec:heart}
\subsection{Basic properties of the heart}\label{ssec:basic_of_heart}
We recall the definition and some basic properties of the heart.
Throughout this section, we fix a triangulated category $\C$ equipped with a cotorsion pair $(\S,\V)$.
For two classes $\U$ and $\V$ of objects in
$\T$, we denote by $\U *\V$ the class of objects $X$ occurring in a triangle $U\to X\to V\to U[1]$
with $U\in\U$ and $V\in\V$.

\begin{definition}
Let $\C$ be a triangulated category equipped with a cotorsion pair in $\C$, and put $\W:=\S\cap\V$.
We define the following associated subcategories:
$$
\C^-:=\S[-1]*\W;\ \C^+:=\W*\V[1];\ \H:=\C^+\cap\C^-.
$$
The additive quotient $\H/\W$ is called the \textit{heart} of $(\S,\V)$.
\end{definition}

Abe and Nakaoka showed the following assertions.

\begin{lemma}\label{lem:coreflection}\cite[Def. 3.5]{AN12}
For any $X\in\C$, there exists the following commutative diagram
\begin{equation}\label{diag:coreflection}
\xymatrix@R=12pt@C=12pt{
V_X\ar[rr]&&U_X\ar[rr]^\alpha&&X\ar[rr]\ar[dr]&&V_X[1]\\
&&&&&V'_X\ar[ur]&
}
\end{equation}
where $V_X,V'_X\in\V$, $\alpha$ is a left $(\C^-)$-approximation of $X$ and the first row is a triangle.
This triangle is called a \textnormal{coreflection triangle} of $X$.
\end{lemma}

By a closer look at the above and its dual, we conclude the above left $(\C^-)$-approximation $\alpha$ gives rise to a functor as below.

\begin{lemma}\cite[Lem. 4.2]{AN12}\label{lem:AN}
The canonical inclusion $\C^-/\W \hookrightarrow\C/\W$ has a left adjoint $L$ which restricts to the functor $L:\C^+/\W\to\H/\W$.
Dually, the canonical inclusion $\C^+/\W \hookrightarrow\C/\W$ has a right adjoint $R$ which restricts to the functor $R:\C^-/\W\to\H/\W$.
Furthermore, there exists a natural isomorphism $\eta:LR\xto{\sim}RL$.
\end{lemma}

The following is their main result.

\begin{theorem}\cite[Thm. 6.4]{Nak11}\cite[Thm. 5.7]{AN12}\label{thm:AN}
The heart $\H/\W$ is abelian.
Moreover, the functor $\mathsf{coh}:=LR\pi:\C\to\H/\W$ is cohomological.
\end{theorem}

Note that the cotorsion class $\S$ admits an extriangulated structure.
For later use, we also recall some consequences of the assumption that $\S$ has enough projectives.
In this case, the heart $\H/\W$ and the cohomological functor $\mathsf{coh}$ have nicer descriptions.

\begin{lemma}\cite[Cor. 3.8]{LN19}\label{lem:kernel_of_cohom}
The kernel of cohomological functor $\mathsf{coh}:\C\to\H/\W$ is $\add\S*\V$, the additive closure of $\S*\V$.
In particular, $\add\S*\V$ is extension-closed in $\C$.
\end{lemma}

The subcategory $\S*\V$ will play important roles in the rest.

\begin{proposition}\cite[Thm. 4.10, Prop. 4.15]{LN19}\label{prop:enough_projective_heart}
The following are equivalent:
\begin{enumerate}
\item[\textnormal{(i)}] $\S$ has enough projectives;
\item[\textnormal{(ii)}] $(\proj(\S), \add(\S*\V))$ forms a cotorsion pair of $\C$.
\end{enumerate}
Under the above equivalent conditions,
there exists an equivalence $\Psi:\H/\W\xto{\sim}\mod(\proj(\S[-1]))$ which makes the following diagram commutative up to isomorphisms
$$
\xymatrix{
\C\ar[rd]_{\Hom(\proj(\S[-1]), -)\ \ }\ar[r]^{\mathsf{coh}}&\H/\W\ar[d]^\Psi\\
&\mod(\proj(\S[-1]))
}
$$
\end{proposition}

Note that this explains why the heart $\H/\W$ is abelian and the associated functor $\mathsf{coh}$ is cohomological.

\subsection{From cotorsion pair to gHTCP}\label{ssec:from_CP_to_gHTCP}
We shall show that the heart in triangulated categories can be obtained as a gHTCP localization.
We consider the pair $\P:=((\S,\C^+),(\C^-,\V))$.
Our result shows that $\P$ is a gHTCP and the heart construction is nothing but the Gabriel-Zisman localization of $\C$ with respect to the class $\mathbb{V}$ associated to $\P$.

\begin{corollary}\label{cor:heart_is_gHTCP}
The pair $\P$ forms a gHTCP. 
Moreover, the following hold.
\begin{enumerate}
\item[\textnormal{(1)}] We have $\Z=\H, \I=\W$ and $\Z/\I=\H/\W$.
\item[\textnormal{(2)}] There exists a class $\mathbb{V}$ of morphisms in $\C$ such that the associated localization $\mathsf{L}_\mathbb{V}:\C\to\C[\mathbb{V}^{-1}]$ induces an equivalence $\Phi:\H/\W\xto{\sim}\C[\mathbb{V}^{-1}]$.
\item[\textnormal{(3)}] We have $\mathsf{L}_\mathbb{V}\cong\Phi\circ\mathsf{coh}$. 
\end{enumerate}
\end{corollary}
\begin{proof}
We shall verify that the pair $\P$ forms a gHTCP.
We firstly show that $\C^-$ is closed under direct summands.
Let $T$ be an object in $\C^-$ together with a decomposition $T\cong T_1\oplus T_2$.
By definition,
there exists a triangle $S[-1]\xto{a}T\xto{b}I\to S$ with $S\in\S$ and $I\in I$.
A canonical projection $p_i:T\to T_i$ induces a triangle $S[-1]\xto{p_ia}T_i\xto{d} X_i\to S$ and a morphism $c:I\to X_i$ with $cb=dp_i$ for $i=1,2$.
Next, we resolve $X_i$ by $(\S,\V)$ to get a triangle $S'[-1]\to X_i\xto{e} V\to S'$ with $S'\in\S$ and $V\in\V$.
We obtain the following commutative diagram from the above triangles:
$$
\xymatrix@R=16pt{
&&S'[-1]\ar@{=}[r]\ar[d]&S'[-1]\ar[d]\\
S[-1]\ar[d]\ar[r]^{p_ia}&T_i\ar@{=}[d]\ar[r]^d&X_i\ar[d]^e\ar[r]&S\ar[d]\\
Y[-1]\ar[r]&T_i\ar[r]&V\ar[r]\ar[d]&Y\ar[d]\\
&&S'\ar@{=}[r]&S'
}
$$
where all rows and columns are triangles.
Note that $Y\in\S$.
If $V\in\I$, this forces $T_i\in\C^-$.
It suffices to check $V\in\S$.
To this end, we consider a morphism $f:V\to V'[1]\in\V[1]$.
The composition $fec:I\to V'[1]$ is zero because of $(\I,\V[1])=0$.
Thus we have $fecb=fedp_i=0$ and hence $fed=0$ which shows $fe$ factors through $S\in\S$.
Moreover, $(\S,\V[1])=0$ forces $fe=0$ and $f$ factors through $S'\in\S$.
We thus conclude $f=0$ and $V\in\S$.
Dually we can confirm that $\C^+$ is also closed under direct summands.

Thanks to Lemma \ref{lem:coreflection} and the dual, it follows that $(\S,\C^+)$ and $(\C^-,\V)$ are a left cotorsion pair and a right one, respectively.

(Hov1): Let $S$ be an object in $\S$.
Since $(\C^-,\V)$ is a right cotorsion pair, we have a conflation $V_S\to S^-\to S$ with $S^-\in\C^-$ and $V_S\in\V$.
The condition $\Ext^1_\C(\S,\V)=0$ shows that $S$ is a direct summand of $S^-$.
Dually we have $\C^+\supseteq \V$.

(Hov2): Consider $X\in\S\cap\C^+$ together with a conflation $V_X\to I_X\to X$ with $V_X\in\V$ and $I_X\in\I$.
Since $\Ext^1_\C(X,V_X)=0$, $X$ belongs to $\I$.
Similarly, we have $\I=\C^-\cap\V$.

Due to Lemma \ref{lem:AN}, we conclude that $\P$ is a gHTCP.
The remaining assertions are obvious.
\end{proof}

We can obtain from Corollary \ref{cor:heart_is_gHTCP} the following picture which is a special case of the right one in (\ref{picture_of_this_article}).
\begin{equation*}\label{picture_of_heart}
\xymatrix@!C=36pt@R=32pt{
&(\S,\V)\ar@{~>}[rd]^{\textnormal{Corollary \ref{cor:heart_is_gHTCP}}}\ar@{~>}[ld]_{\textnormal{Abe-Nakaoka}}&\\
\H/\W\ar@{->}[rr]^\sim&&\C[\mathbb{V}^{-1}]
}
\end{equation*}

Combining Corollary \ref{cor:heart_is_gHTCP} and Proposition \ref{prop:enough_projective_heart}, we recover the following result which can be deduced from \cite[Thm 4.4]{BM13a} and \cite[Prop 6.2(3)]{IY08}.

\begin{corollary}
Let $k$ be a field.
We consider a $k$-linear $\Hom$-finite Krull-Schmidt triangulated
category $\C$ with a rigid object $T$.
\begin{enumerate}
\item[\textnormal{(1)}] 
There exists a class $\mathbb{V}$ of morphisms and an equivalence $\C[\mathbb{V}^{-1}]\xto{\sim}\mod\End_\C(T)^{\op}$.
\item[\textnormal{(2)}] 
There exists an equivalence $\frac{\add T[-1]*\add T}{\add T}\xto{\sim}\mod\End_\C(T)$.
\end{enumerate}
In particular, we have an equivalence $\Phi:\frac{\add T[-1]*\add T}{\add T}\xto{\sim}\C[\mathbb{V}^{-1}]$.
\end{corollary}
\begin{proof}
Due to the $\Hom$-finiteness of $\C$ and the rigidity of $T$, we have a cotorsion pair $(\add T,(\add T)^{\perp_1})$ of $\C$.
Applying Corollary \ref{cor:heart_is_gHTCP} to the cotorsion pair, we have a gHTCP $((\add T,\C),(\add T[-1]*\add T,(\add T)^{\perp_1}))$ and a desired equivalence $\Phi$.
\end{proof}

\subsection{A factorization through a preabelian category}\label{ssec:factorization}
Keeping in mind Buan-Marsh's localizations, we shall subsequently show that the above Gabriel-Zisman localization $\mathsf{L}_\mathbb{V}:\C\to\H/\W$ is not far from one admitting calculus of left and right fractions.
The following our main result is formulated under functorial finiteness of $\S*\V$.

\begin{theorem}\label{thm:generalized_BM2}
Let $\C$ be a triangulated
category equipped with a cotorsion pair $(\S,\V)$.
Assume that $\S*\V$ is functorially finite in $\C$.
Then,
\begin{enumerate}
\item[\textnormal{(1)}] the additive quotient $\C/(\S*\V)$ is preabelian;
\item[\textnormal{(2)}] the class $\mathbb{R}$ of regular morphisms in $\C/(\S*\V)$ admits a calculus of left and right fractions;
\item [\textnormal{(3)}] there exists the following commutative diagram up to isomorphism and a factorization of $\mathsf{L}_\mathbb{V}\cong \mathsf{L}_\mathbb{R}\circ\varpi$.
\[
\xymatrix{
&\H/\W\ar[rd]^{\Phi}&\\
\C\ar[ru]^{\mathsf{coh}}\ar[rd]_{\varpi}\ar[rr]^{\mathsf{L}_\mathbb{V}}&&\C[\mathbb{V}^{-1}],\\
&\C/(\S*\V)\ar[ru]_{\mathsf{L}_\mathbb{R}}&
}
\]
where $\varpi$ is the natural additive quotient functor;
\item [\textnormal{(4)}] the localization functor $\mathsf{L}_\mathbb{R}:\C/(\S*\V)\to(\C/(\S*\V))[\mathbb{R}^{-1}]$ induces an equivalence
$$(\C/(\S*\V))[\mathbb{R}^{-1}]\xto{\sim} \H/\W .$$
\end{enumerate}
\end{theorem}

The rest of this section will be occupied to prove the theorem.
Our method is similar to Buan-Marsh's one which strongly depends on Rump's localization theory of preabelian category.
To show Theorem \ref{thm:generalized_BM2}, we firstly show that $\C/(\S*\V)$ is an integral preabelian category.
Put $\K:=\add\S*\V$ and identify $\C/(\S*\V)$ with $\C/\K$.

\begin{proposition}\label{prop:C/K_is_preabelian}
The additive quotient $\C/(\S*\V)$ is preabelian.
\end{proposition}
To show the above, we prepare the following auxiliary triangle to construct cokernels in $\C/\K$.

\begin{lemma}\label{lem:to_construct_cokernel}
For any object $X\in\C$, there exists a triangle
\begin{equation}\label{seq:aux_tri}
K'[-1]\xto{p} X\xto{\iota} \tilde K\to K'
\end{equation}
in $\C$ such that $\tilde K,K'\in\K$ and $\iota$ is a left $\K$-approximation of $X$.
\end{lemma}
\begin{proof}
Resolve $X$ by the cotorsion pair $(\S,\V)$ to get a triangle $S[-1]\xto{a} X\to V\to S$ with $S\in\S$ and $V\in\V$.
Since $\K$ is functorially finite, there exists a left $\K$-approximation $S[-1]\xto{b} K$ of $S[-1]$ which completes a triangle $K[-1]\to K'[-1]\xto{c} S[-1]\xto{b} K$.
Since Lemma \ref{lem:kernel_of_cohom} says $\K$ is extension-closed, $K'$ belongs to $\K$.
By the octahedral axiom, the triangles appearing so far induce the following commutative diagram:
$$
\xymatrix@R=16pt{
&&K\ar[d]\ar@{=}[r]&K\ar[d]\\
K'[-1]\ar[r]^{ac}\ar[d]_c&X\ar[r]^d\ar@{=}[d]&\tilde K\ar[r]\ar[d]&K'\ar[d]^{c[1]}\\
S[-1]\ar[r]^a&X\ar[r]&V\ar[d]\ar[r]&S\ar[d]^{b[1]}\\
&&K[1]\ar@{=}[r]&K[1]
}
$$
with all rows and columns are triangles.
Again, since $\K$ is extension-closed, we have $\tilde K\in\K$.
The second row is a desired one.
In fact, for any morphism $\alpha:X\to L\in\K$, the composition $\alpha a$ factors through the left $\K$-approximation $b$.
It follows that $\alpha ac=0$ and $\alpha$ factors through $d$.
\end{proof}

\begin{proof}[Proof of Proposition \ref{prop:C/K_is_preabelian}]
Let $X\xto{f} Y\xto{g} Z\to X[1]$ be a triangle in $\C$.
We shall provide a construction of the cokernel of $\varpi(f)$ in $\C/\K$.
For the object $X$, we consider an auxiliary triangle (\ref{seq:aux_tri}).
The octahedral axiom gives the following commutative diagram:
\[
\xymatrix@R=16pt{
&&\tilde K\ar[d]\ar@{=}[r]&\tilde K\ar[d]\\
K'[-1]\ar[r]^{fp}\ar[d]_{p}&Y\ar@{=}[d]\ar[r]^h&Z'\ar[d]\ar[r]&K'\ar[d]^{p[1]}\\
X\ar[r]^f&Y\ar[r]^g&Z\ar[d]\ar[r]&X[1]\ar[d]^{\iota[1]}\\
&&\tilde K[1]\ar@{=}[r]&\tilde K[1]
}
\]
The image of the sequence $X\xto{f}Y\xto{h}Z'$ is a cokernel sequence in $\C/\K$.
To verify this, we take a morphism $g_0:Y\to M$ such that $g_0f$ factors through an object in $\K$.
Since $\iota:X\to \tilde K$ is a left $\K$-approximation, $g_0f$ factors through $\iota$.
Therefore we have $g_0fp=0$ and that there exists a morphism $g_1:Z'\to M$ together with $g_1h=g_0$.
To show the uniqueness of $g_1$, we consider a morphism $g_2:Z'\to M$ together with $\varpi(g_2h)=\varpi(g_0)$.
Then $(g_1-g_2)h$ factors through an object $K_0\in\K$.
More precisely, there exists morphisms $Y\xto{x}K_0\xto{y}Z$ with $yx=(g_1-g_2)h$.
Taking a weak pushout of $h$ along $x$ we get the following commutative diagram made of triangles.
\[
\xymatrix@R=16pt{
K'[-1]\ar[r]^{fp}\ar@{=}[d]&Y\ar@{}[rd]|{\textnormal{(Po)}}\ar[d]_x\ar[r]^h&Z'\ar[d]\ar[r]&K'\ar@{=}[d]\\
K'[-1]\ar[r]&K_0\ar[r]&K_1\ar[r]&K'
}
\]
Note that $K_1$ also belongs to $\K$.
Hence the commutative square of $yx=(g_1-g_2)h$ shows that $g_1-g_2$ factors through $K_1$.
We have thus concluded that $\C/\K$ has cokernels.
Dually we can construct the kernel.
\end{proof}

By using this construction of the cokernel, we characterize epimorphisms in $\C/\K$.

\begin{lemma}\label{lem:char_of_epi}
Let $X\xto{f} Y\xto{g}Z\to X[1]$ be a triangle in $\C$.
Then, the morphism $\varpi(f)$ is an epimorphism in $\C/\K$ if and only if the morphism $g: Y\to Z$ factors through an object in $\K$.
\end{lemma}
\begin{proof}

[A] Firstly, we consider the case of $Z\in\K$.
We shall show that $\varpi (f)$ is epic in $\C/\K$.
To construct a cokernel of $\varpi(f)$, we take an auxiliary triangle (\ref{seq:aux_tri}) $K[-1]\xto{p}X\xto{\iota}\tilde K\to K$ and complete the following commutative diagram by the octahedral axiom:
\begin{equation}\label{diag:construction_of_cokernels}
\xymatrix@R=16pt{
&&\tilde K\ar[d]\ar@{=}[r]&\tilde K\ar[d]\\
K[-1]\ar[r]^{fp}\ar[d]_{p}&Y\ar@{=}[d]\ar[r]^{g'}&Z'\ar[d]^h\ar[r]&K\ar[d]^{p[1]}\\
X\ar[r]^f&Y\ar[r]^g&Z\ar[d]\ar[r]&X[1]\ar[d]^{\iota[1]}\\
&&\tilde K[1]\ar@{=}[r]&\tilde K[1]
}
\end{equation}
Note that $Z'\in\K$.
Since the image of the sequence $X\xto{f}Y\xto{g'}Z'\to 0$ is a cokernel sequence in $\C/\K$,
$\varpi(Z')=0$ forces the morphism $\varpi(f)$ to be epic.

[B] Assume that $\varpi(f)$ is an epimorphism.
To construct the cokernel of $\varpi(f)$, we take an auxiliary triangle (\ref{seq:aux_tri}) and complete the same diagram (\ref{diag:construction_of_cokernels}).
Thus we have a cokernel sequence $X\xto{\varpi(f)}Y\xto{\varpi(g')}Z'\to 0$ in $\C/\K$.
Since $\varpi(f)$ is epic, we get $Z'\in\K$.
Hence $g:Y\to Z$ factors through $Z'\in\K$.

To show the converse, we consider a triangle $X\xto{f}Y\xto{g}Z\to X[1]$ where $g$ factors through an object in $\K$.
To construct the cokernel of $\varpi(f)$, we consider the same diagram as (\ref{diag:construction_of_cokernels}).
Since $g$ factors through an object in $\K$, we have $\varpi(g)=\varpi(hg')=0$.
Since $\cocone(h)\cong \tilde K\in\K$, by the dual of [1], we have that $\varpi(h)$ is a monomorphism.
Thus $\varpi(g')=0$ which forces $\varpi(f)$ is an epimorphism.
\end{proof}

We now verify a nicer property.

\begin{proposition}\label{prop:C/K_is_integral}
The preabelian category $\C/(\S*\V)$ is integral.
\end{proposition}
\begin{proof}
We consider the following pullback diagram in $\C/\K$
$$
\xymatrix@R=16pt@C=16pt{
A\ar[d]_{\varpi(b)}\ar[r]^{\varpi(a)}&B\ar[d]^{\varpi(c)}\\
C\ar@{->>}[r]^{\varpi(d)}&D
}
$$
with $\varpi(d):C\to D$ epic.
We shall show that $\varpi(a)$ is epic.
Complete a triangle $C\xto{d}D\xto{e} E\to C[1]$ in $\C$.
By Lemma \ref{lem:char_of_epi}, the morphism $e$ is factored as $e:D\xto{e_2} L\xto{e_1} E$ with $L\in\K$.
For the object $B\in\C$, take a triangle $K[-1]\xto{p}B\xto{\iota}\tilde K\to K$ (\ref{seq:aux_tri}).
Since $\iota:B\to \tilde K$ is a left $\K$-approximation, we have a morphism $f:\tilde K\to L$ which makes the following diagram in $\C$ commutative
$$
\xymatrix@R=10pt@C=15pt{
K[-1]\ar@{..>}[dd]_g\ar[rr]^p&&B\ar[dd]_c\ar[rr]^\iota&&\tilde K\ar[dd]^{e_1f}\ar[dl]|f\ar[rr]&&K\ar@{..>}[dd]\\
&&&L\ar[rd]|{e_1}&&&\\
C\ar[rr]_d&&D\ar[ru]|{e_2}\ar[rr]_e&&E\ar[rr]&&C[1]
}
$$
Then we get the above dotted arrow $g:K[-1]\to C$.
By the universality of the pullback, we have a further commutative diagram in $\C/(\S*\V)$:
$$
\xymatrix@R=16pt@C=16pt{
K[-1]\ar@/^1.5pc/[rrd]^{\varpi(p)}\ar@/_1.5pc/[ddr]_{\varpi(g)}\ar@{..>}[rd]&&\\
&A\ar[d]_{\varpi(b)}\ar[r]^{\varpi(a)}&B\ar[d]^{\varpi(c)}\\
&C\ar[r]^{\varpi(d)}&D
}
$$
Since $\varpi(p)$ is an epimorphism by Lemma \ref{lem:char_of_epi}, so is $\varpi(a)$.
\end{proof}

Thanks to Proposition \ref{prop:localization_of_integral}, we have the following localization admits a calculus of left and right fractions.

\begin{corollary}\label{cor:preabel_to_abel}
The class $\mathbb{R}$ of regular morphisms admits a calculus of left and right fractions.
The Gabriel-Zisman localization $(\C/\K)[\mathbb{R}^{-1}]$ of $(\C/\K)$ with respect to $\mathbb{R}$ is abelian.
\end{corollary}

Next we shall show that the abelian category $(\C/\K)[\mathbb{R}^{-1}]$ is equivalent to the heart.

\begin{proposition}\label{prop:localization_VR}
There exists an equivalence $F:\C[\mathbb{V}^{-1}]\xto{\sim} (\C/\K)[\mathbb{R}^{-1}]$ which makes the following diagram commutative up to isomorphism
$$
\xymatrix{
\C\ar[d]_{\mathsf{L}_\mathbb{V}}\ar[r]^\varpi&\C/\K\ar[r]^{\mathsf{L}_\mathbb{R}\ \ \ }&(\C/\K)[\mathbb{R}^{-1}]\\
\C[\mathbb{V}^{-1}]\ar[rru]_F&&
}
$$
\end{proposition}
\begin{proof}
To show the existence of $F$, we shall confirm $\mathsf{L}_\mathbb{R}\circ\varpi(\mathbb{V})$ is a class of isomorphims in $(\C/\K)[\mathbb{R}^{-1}]$.
Consider a morphism $f\in\udef$ together with a triangle $V\to U\xto{f}X\xto{g} V[1]$ where $V\in\V$ and $f$ is a right $\U$-approximation of $X$.
By comparing the above triangle and the coreflection triangle (\ref{diag:coreflection}), we have that $g$ factors through $V'_X\in\V$.
Due to Lemma \ref{lem:char_of_epi} and its dual, we conclude that $\varpi(f)$ is a regular morphism in $\C/\K$.
Similarly, $\varpi(\tinf)$ forms a class of regular morphisms in $\C/\K$. 
By the universality of $\C[\mathbb{V}^{-1}]$, the existence and uniqueness of $F$ follow.

Next, we shall construct the inverse of $F$.
Firstly we recall $R\pi(\V)=L\pi(\S)=0$ in $\Z/\I$.
Hence, since $\mathsf{L}_\mathbb{V}$ is cohomological, we also have $\mathsf{L}_\mathbb{V}(\K)=0$.
Therefore we have an additive functor $G':\C/\K\to\C[\mathbb{V}^{-1}]$ such that $\mathsf{L}_\mathbb{V}\cong G'\circ\varpi$.
It suffices to show that $G'$ sends any regular morphism to an isomorphism in $\C[\mathbb{V}^{-1}]$.
Due to Lemma \ref{lem:char_of_epi}, again since $\mathsf{L}_\mathbb{V}$ is cohomological, $G'(\mathbb{R})$ is a class of isomorphisms.
By the universality, the existence of a desired functor $G$ follows and $F$ is an equivalence.
\end{proof}

Now we are ready to prove Theorem \ref{thm:generalized_BM2}.

\begin{proof}[Proof of Theorem \ref{thm:generalized_BM2}]
The assertions (1) and (2) follow from Proposition \ref{prop:C/K_is_preabelian} and Corollary \ref{cor:preabel_to_abel}.
The others (3) and (4) follow from Proposition \ref{prop:localization_VR} and Corollary \ref{cor:heart_is_gHTCP}.
\end{proof}

As an application of Theorem \ref{thm:generalized_BM2}, we conclude that the cohomological functor $\mathsf{coh}:\C\to\H/\W$ has a universality in the following sense.

\begin{corollary}\label{cor:universal_cohomological_functor}
Assume that $\S*\V$ is functorially finite in $\C$.
Let $H$ be a cohomological functor from $\C$ to an abelian category $\A$.
If $H(\K)=0$, then there, uniquely up to isomorphism, exists an exact functor $H'$ such that $H\cong H'\circ\mathsf{coh}$. 
\end{corollary}
\begin{proof}
By $H(\K)=0$, we have a functor $H'':\C/\K\to \A$ which makes the following diagram commutative up to isomorphism
$$
\xymatrix{
\C\ar[d]_H\ar[r]^{\varpi\ \ }&\C/\K\ar[ld]|{H''}\ar[r]^{\mathsf{L}_\mathbb{R}\ \ \ }&(\C/\K)[\mathbb{R}^{-1}]\ar@{..>}[lld]^{H'}\\
\A&&
}
$$
Since $H$ is cohomological, $H''$ sends $\mathbb{R}$ to a class of isomorphisms in $\A$. Thus we have a desired functor $H'$.
The exactness directly follows from the construction of cokernels and kernels in $(\C/\K)[\mathbb{R}^{-1}]$ (see the proof of Proposition \ref{prop:C/K_is_preabelian}).
\end{proof}

\section{Comments on another related work}\label{sec:related_work}
In this section, we give some comments on results closely related to Theorem \ref{thm:generalized_BM2} which can be easily deduced from \cite{Nak13, HS20} (see \cite{Liu13} for exact case).
Their method depends on a use of the twin cotorsion pair $\Q:=((\S,\T),(\U,\V))$ which is defined to be a pair of two cotorsion pairs with $\S\subseteq \U$.

\begin{remark}
\begin{enumerate}
\item The twin cotorsion pair $\Q=((\S,\T),(\U,\V))$ does not require that $(\S,\V)$ forms a cotorsion pair.
If this is the case, it forces $\S=\U$ and $\T=\V$.
\item Note that our gHTCP is not a twin cotorsion pair in general.
\end{enumerate}
\end{remark}

The \emph{heart} $\underline{\H}_{\Q}$ of twin cotorsion pair $\Q$ was introduced as a generalization of the heart of a cotorsion pair \cite[Def. 2.8]{Nak13}.
Their results show that the heart $\underline{\H}_{\Q}$ of some special twin cotorsion pair is an integral category, and its localization $\underline{\H}_{\Q}[\mathbb{R}^{-1}]$ with respect to the class of regular morphisms is equivalent to the heart $\underline{\H}_{(\S,\T)}$ of the cotorsion pair $(\S,\T)$.

By Proposition \ref{prop:enough_projective_heart} and its dual, we have the following corollary which shows that, under some assumption, a cotorsion pair $(\S,\V)$ gives rise to a twin cotorsion pair.

\begin{corollary}
Let $\C$ be a triangulated category equipped with a cotorsion pair $(\S,\V)$.
If $\S$ has enough projectives $\mathsf{P}(\S)$ and $\V$ has enough injectives $\mathsf{I}(\V)$,
then we have a twin cotorsion pair 
$$\Q:=((\proj(\S),\K),(\K,\inj(\V))),$$
where $\K=\add(\S*\V)$.
Moreover, there exists an equivalence between the heart $\H/\W$ of $(\S,\V)$ and that of $(\proj(\S),\K)$.
\end{corollary}

In this case, it directly follows from \cite[Thm. 6.3]{Nak13} and \cite[Thm. 5.6]{HS20} that there exists a commutative diagram which is a special case of Theorem \ref{thm:generalized_BM2}(3) as below:
$$
\xymatrix@C=30pt{
\C\ar[rd]\ar[rr]^{\mathsf{coh}}&&\H/\W\\
&\underline{\H}_{\Q}\ar[ru]_{\mathsf{L}_\mathbb{R}}&
}
$$
where $\mathsf{L}_\mathbb{R}$ denotes the localization with respect to the class $\mathbb{R}$ of regular morphisms.

\subsection*{Data Availability Statement}
N/A

\subsection*{Acknowledgements}
The author wishes to thank Professor Hiroyuki Nakaoka for giving valuable comments.
In addition he thanks the anonymous referees for their carefully reading the manuscript and
helpful suggestions.



\begin{thebibliography}{99}

\bibitem[AN]{AN12}
N.~Abe, H.~Nakaoka, \emph{General heart construction on a triangulated category (II): Associated homological functor}.
Appl. Categ. Structures 20 (2012), no. 2, 161--174.

\bibitem[Aus]{Aus66}
M.~Auslander, \emph{Coherent functors}.
1966 Proc. Conf. Categorical Algebra (La Jolla, Calif., 1965) pp. 189--231 Springer, New York.

\bibitem[BBD]{BBD}
A.~Beilinson, J.~Bernstein, P.~Deligne, \emph{Faisceaux Pervers (Perverse sheaves)}.
Analysis and Topology on Singular Spaces, I, Luminy, 1981, Asterisque 100 (1982) 5--171 (in French).

\bibitem[Bel]{Bel13}
A.~Beligiannis,
\emph{Rigid objects, triangulated subfactors and abelian localizations}.
274 (2013), no. 3-4, 841--883.


\bibitem[BMPS]{BMPS19}
V.~Becerril, O.~Mendoza, M.~P\'erez, V.~Santiago, \emph{Frobenius pairs in abelian categories. Correspondences with cotorsion pairs, exact model categories, and Auslander-Buchweitz contexts}.
J. Homotopy Relat. Struct. 14 (2019), no. 1, 1--50.

\bibitem[BM1]{BM13a}
A.~Buan, R.~Marsh, \emph{From triangulated categories to module categories via localisation}.
Trans. Amer. Math. Soc. 365 (2013), no. 6, 2845--2861.

\bibitem[BM2]{BM13b}
A.~Buan, R.~Marsh, \emph{From triangulated categories to module categories via localization II: calculus of fractions}.
J. Lond. Math. Soc. (2) 86 (2012), no. 1, 152--170.

\bibitem[DS]{DS95}
W.~G.~Dwyer, J.~Spali\'nski, \emph{Homotopy theories and model categories}.
Handbook of algebraic topology, North-Holland, Amsterdam, 1995, pp. 73--126.

\bibitem[Fri]{Fri08}
T.~Fritz, \emph{Categories of Fractions Revisited}.
arXiv:0803.2587v2.

\bibitem[GZ]{GZ67}
P.~Gabriel, M.~Zisman, \emph{Calculus of fractions and homotopy theory}.
Ergebnisse der Mathematik und ihrer Grenzgebiete, Band 35 Springer-Verlag New York, Inc., New York 1967 {\rm x}+168 pp.

\bibitem[Gil]{Gil11}
J.~Gillespie, \emph{Model structures on exact categories}.
J. Pure Appl. Algebra 215 (2011) no. 12, 2892--2902.

\bibitem[Hap]{Hap88}
D.~Happel, \emph{Triangulated categories in the representation theory of finite dimensional algebras}.
London Mathematical Society Lecture Note Series, 119. Cambridge University Press, Cambridge, 1988.

\bibitem[HS]{HS20}
S.~Hassoun, A.~Shah, \emph{Integral and quasi-abelian hearts of twin cotorsion pairs on extriangulated categories}.
Comm. Algebra 48 (2020), no. 12, 5142--5162.

\bibitem[Hov1]{Hov99}
M.~Hovey, \emph{Model categories}.
volume 63 of Mathematical Surveys and Monographs.
American Mathematical Society, Providence, RI, 1999.

\bibitem[Hov2]{Hov02}
M.~Hovey, \emph{Cotorsion pairs, model category structures, and representation theory}.
Math. Z. 241 (2002) no. 3, 553--592.

\bibitem[IYa]{IYa17}
O.~Iyama, D.~Yang, \emph{Quotients of triangulated categories and Equivalences of Buchweitz, Orlov and Amiot--Guo--Keller}.
arXiv:1702.04475.

\bibitem[IYo]{IY08}
O.~Iyama, Y.~Yoshino, \emph{Mutation in triangulated categories and rigid Cohen-Macaulay modules}.
Invent. Math. 172 (2008), no. 1, 117--168.

\bibitem[KR]{KR07}
B.~Keller, I.~Reiten, \emph{Cluster-tilted algebras are Gorenstein and stably Calabi-Yau}.
Adv. Math. 211 (2007), no. 1, 123--151.

\bibitem[KZ]{KZ08}
S.~Koenig, B.~Zhu, \emph{From triangulated categories to abelian categories: cluster tilting in a general framework}.
Math. Z. 258 (2008), no. 1, 143--160.

\bibitem[Len]{Len98}
H.~Lenzing, \emph{Auslander's work on Artin algebras}.
in Algebras and modules, I (Trondheim,
1996), 83--105, CMS Conf. Proc., 23, Amer. Math. Soc., Providence, RI, 1998.

\bibitem[Li]{Li16}
Z.-W.~Li, \emph{The realization of Verdier quotients as triangulated subfactors}.
arXiv:1612.08340v5.

\bibitem[Liu2]{Liu13}
Y.~Liu, \emph{Hearts of twin cotorsion pairs on exact categories}.
J. Algebra 394 (2013), 245--284.

\bibitem[LN]{LN19}
Y.~Liu, H.~Nakaoka, \emph{hearts of twin cotorsion pairs on extriangulated categories}.
J. Algebra 528 (2019), 96--149.

\bibitem[Nak1]{Nak11}
H.~Nakaoka, \emph{General heart construction on a triangulated category (I): Unifying $t$-structures and cluster tilting subcategories}.
Appl. Categ. Structures 19 (2011), no. 6, 879--899.

\bibitem[Nak2]{Nak13}
H.~Nakaoka, \emph{General heart construction for twin torsion pairs on triangulated categories}.
J. Algebra 374 (2013), 195--215.

\bibitem[Nak3]{Nak18}
H.~Nakaoka, \emph{A simultaneous generalization of mutation and recollement of cotorsion pairs on a triangulated category}.
Appl. Categ. Structures 26 (2018), no. 3, 491--544.

\bibitem[NP]{NP19}
H.~Nakaoka, Y.~Palu, \emph{Extriangulated categories, Hovey twin cotorsion pairs and model structures}.
Topol. G\'eom. Diff\'er. Cat\'eg, vol LX (2019), Issue 2, 117-193.

\bibitem[Nee]{Nee01}
A.~Neeman, \emph{Triangulated categories}.
Annals of Mathematics Studies 148, Princeton University Press (2001).

\bibitem[Pal]{Pal14}
Y.~Palu, \emph{From triangulated categories to module categories via homotopical algebra}.
arXiv:1412.7289.

\bibitem[Pop]{Pop73}
N.~Popescu, \emph{Abelian categories with applications to rings and modules}.
London Mathematical Society Monographs, No. 3. Academic Press, London-New York, 1973. {\rm xii}+467 pp.

\bibitem[Psa]{Psa14}
C.~Psaroudakis, \emph{Homological theory of recollements of abelian categories}.
J. Algebra 398 (2014), 63--110.

\bibitem[Rum]{Rum01}
W.~Rump, \emph{Almost abelian categories}.
Topologie G\'eom. Diff\'erentielle Cat\'eg. 42 (2001), no. 3, 163--225.

\bibitem[Sal]{Sal79}
L.~Salce, \emph{Cotorsion theories for abelian groups}
in Symposia Mathematica,
Vol. XXIII (Conf. Abelian Groups and their Relationship
to the Theory of Modules, INDAM, Rome, 1977) (1979), 11--32.

\end{thebibliography}
\end{document}